\documentclass[11pt,reqno]{amsart}
\title[Perverse schobers, stability conditions and quadratic differentials II]{Perverse schobers, stability conditions and quadratic differentials II: relative graded Brauer graph algebras}

\author{Merlin Christ}
\address{MC: Université Paris Cité and Sorbonne Université, CNRS, IMJ-PRG, F-75013 Paris, France.}
\email{merlin.christ@imj-prg.fr}
\thanks{
M.C. acknowledges support by the Deutsche Forschungsgemeinschaft under Germany’s Excellence Strategy – EXC 2121 “Quantum Universe” – 390833306. M.C. has received funding from the European Union’s Horizon 2020 research and innovation programme under the Marie Skłodowska-Curie grant agreement No 101034255.}

\author{Fabian Haiden}
\address{FH: Centre for Quantum Mathematics, Department of Mathematics and Computer Science, University of Southern Denmark, Campusvej 55, 5230 Odense, Denmark}
\email{fab@sdu.dk}
\thanks{
F.H. is supported by the VILLUM FONDEN, VILLUM Investigator grant 37814 and the Sapere Aude grant 3120-00076B from the Independent Research Fund Denmark (DFF).
This paper is partly a result of the ERC-SyG project Recursive and Exact New Quantum Theory (ReNewQuantum) which received funding from the European Research Council (ERC) under the European Union's Horizon 2020 research and innovation programme under grant agreement No 810573.
}

\author{Yu Qiu}
\address{Qy:
	Yau Mathematical Sciences Center and Department of Mathematical Sciences,
	Tsinghua University,
    100084 Beijing,
    China.
    \&
    Beijing Institute of Mathematical Sciences and Applications, Yanqi Lake, Beijing, China}
\email{yu.qiu@bath.edu}
\thanks{
Qy is supported by
National Natural Science Foundation of China (No. 12425104, 12031007 and 12271279) and
National Key R\&D Program of China (No. 2020YFA0713000).}

\usepackage[margin=2.9cm]{geometry}

\usepackage{float}

\usepackage[colorlinks, linkcolor=blue,anchorcolor=blue,
    citecolor=blue,urlcolor=cyan, bookmarksopen,bookmarksdepth=2]{hyperref}
\usepackage[usenames,dvipsnames]{xcolor}
\usepackage{enumerate}
\usepackage{subfigure}
\usepackage{bookmark}
\usepackage{url}
\usepackage{ifthen}
\usepackage{cleveref}
\usepackage{enumitem}
\usepackage{array}   
\newcolumntype{L}{>{$}l<{$}} 

\usepackage{tikz}\usetikzlibrary{matrix}
\usepackage{tikz-cd}

\usetikzlibrary{matrix, calc, arrows, shapes,
decorations.pathreplacing, decorations.markings, decorations.pathmorphing,
backgrounds,fit,positioning,shapes.symbols,chains,shadings,fadings}

\tikzset{->-/.style={decoration={  markings,  mark=at position #1 with
    {\arrow{>}}},postaction={decorate}}}
\tikzset{-<-/.style={decoration={  markings,  mark=at position #1 with
    {\arrow{<}}},postaction={decorate}}}

\usepackage{amsfonts}
\usepackage{amsmath,amssymb,amsthm,amsxtra}
\usepackage{mathtools} 
\usepackage{extarrows} 
\usepackage{todonotes} 

\theoremstyle{plain}
\newtheorem{theorem}{Theorem}[section]
\newtheorem{lemma}[theorem]{Lemma}
\newtheorem{corollary}[theorem]{Corollary}
\newtheorem{proposition}[theorem]{Proposition}

\theoremstyle{definition}
\newtheorem{definition}[theorem]{Definition}
\newtheorem{example}[theorem]{Example}
\newtheorem{remark}[theorem]{Remark}
\newtheorem{construction}[theorem]{Construction}

\numberwithin{equation}{section}
\numberwithin{figure}{section}

\def\be{\begin{equation}}
\def\ee{\end{equation}}

\newcommand\hua{\mathcal}

\newcommand\ZZ{\mathbb{Z}}

\renewcommand{\setminus}{\smallsetminus}
\renewcommand{\emptyset}{\varnothing}

\newcommand\noloc{
  \nobreak
  \mspace{6mu plus 1mu}
  {:}
  \nonscript\mkern-\thinmuskip
  \mathpunct{}
  \mspace{2mu}
}

\def\on{\operatorname} 
\mathchardef\mhyphen="2D

\newcommand\Hom{\operatorname{Hom}}

\newcommand{\C}{\operatorname{\hua{C}}} 
\newcommand{\D}{\operatorname{\hua{D}}} 
\newcommand{\per}{\operatorname{per}} 

\newcommand\Stab{\operatorname{Stab}} 

\newcommand{\EG}{\operatorname{EG}} 

\renewcommand{\k}{\mathbf{k}}

\usepackage[all]{xy}

\def\wt{\mathbf{w}}
\newcommand\surf{\mathbf{S}}  
\newcommand\sow{\surf_\wt}  

\newcommand\M{\mathbf{M}} 
\newcommand\W{\Delta} 
\newcommand\A{\Gamma} 

\newcommand{\FQuad}[2]{\operatorname{FQuad}^{#1}(#2)}

\tikzcdset{arrow style=tikz, diagrams={>=stealth}}

\newcommand*\cocolon{
        \nobreak
        \mskip6mu plus1mu
        \mathpunct{}%
        \nonscript
        \mkern-\thinmuskip
        {:}%
        \mskip2mu
        \relax
}

\def\hF{\hua{F}}

\def\hC{\hua{C}}

\newcommand\AS{\mathbb{A}} 
\newcommand\dAS{\AS^*} 
\newcommand\Sgh{\mathbb{S}}
\newcommand\eS{\widehat{\Sgh}}
\def\nn{node{$\bullet$}}
\def\ww{node[white]{$\bullet$}node[red]{$\circ$}}

\newcommand\rgraph{{\bf G}} 
\newcommand\glsec{{\Gamma}} 
\newcommand\losec{\Gamma_{\rm{loc}}} 
\newcommand\CS{\hC({\Sgh,\hF})} 
\newcommand\CSh{\CS^\heartsuit} 
\def\ASG{\A_\Sgh} 

\def\spider{\on{\mathbf{Sp}}}

\def\GSn{G(\Sgh,n)}
\def\QSn{Q(\Sgh,n)}

\def\Ende{\on{End}_{L}}
\def\Endv{\on{End}_v}
\def\psG{\hF_{\Sgh,G}}

\begin{document}
\begin{abstract}
We introduce a class of dg-algebras which generalize the classical Brauer graph algebras. They are constructed from mixed-angulations of surfaces and often admit a (relative) Calabi--Yau structure. We discovered these algebras through two very distinct routes, one involving perverse schobers whose stalks are cyclic quotients of the derived categories of relative Ginzburg algebras, and another involving deformations of partially wrapped Fukaya categories of surfaces.
Applying the results of our previous work \cite{CHQ23}
, we describe the spaces of stability conditions on the derived categories of these algebras in terms of spaces of quadratic differentials.

\end{abstract}
\maketitle
\vspace{-1.4em}
\tableofcontents\addtocontents{toc}{\protect\setcounter{tocdepth}{2}}

\section{Introduction}

\subsection{Relative graded Brauer graph algebras}

Brauer graph algebras, as introduced by Donovan--Freislich~\cite{donovan_freislich}, are certain quiver algebras with relations constructed from the data of a ribbon graph together with a multiplicity $m\geq 1$ for each vertex. We refer the reader to~\cite{schroll_bga} for a survey of more recent developments.
We consider generalizations of these algebras, the \textit{relative graded Brauer graph (=RGB) algebras}. RGB algebras are also constructed from a decorated ribbon graph, called the \textit{S-graph}, but include as input an additional integer $n>0$. Our motivation for considering these generalizations are:
\begin{enumerate}
    \item \textit{Gradings} allow us to construct new examples of Calabi--Yau algebras and triangulated categories of Calabi--Yau dimension equal to $n>0$.
    \item \textit{Relativeness} enables gluing arguments: up to Morita equivalence, any RGB algebra embeds fully faithfully into the fiber product of elementary \textit{relative} graded Brauer graph algebras. 
\end{enumerate}
In the non-relative case we are simply considering $\ZZ$-graded versions of the usual Brauer graph algebras, which are also studied in upcoming work of Gnedin--Opper--Zvonareva.
In the relative case, the RGB algebra has a non-trivial differential.

An S-graph arises from a \textit{weighted marked surface} $\sow$, equipped with a \textit{mixed-angulation}, see \Cref{subsec:WMS}.
The underlying surface ${\bf S}$ of $\sow$ is compact, oriented and possibly with boundary. A mixed-angulation is a decomposition of ${\bf S}$ into polygons with vertices at the marked points.
Each such polygon with $m\geq 1$ edges contains a so-called singular point $x$ in its interior, which we assign the degree $m$.
Such mixed-angulations generalize triangulations and $n$-angulations, where each singular point has degree $3$ and $n$, respectively.
We additionally allow singular points on the boundary of degree $\infty$ (corresponding to $\infty$-gons of the mixed-angulation).
Dual to the mixed-angulation is a graph $\Sgh$ called the S-graph.
The vertices of the S-graph are the singular points, and a degree $m$ singular point has valency $r\leq m$. Some examples can be found in \Cref{fig:S-graphs}.

\begin{figure}[ht]\centering
\makebox[\textwidth][c]{
 \includegraphics[width=15cm]{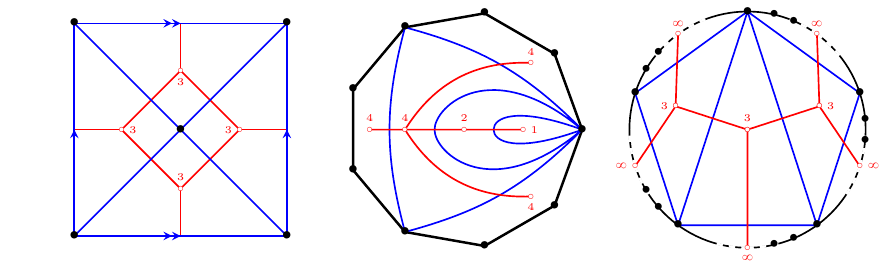}\qquad\quad}
\caption{Three examples of weighted marked surfaces with a mixed-angulation (blue) and dual S-graph (red). The left one is an ideal triangulation of a torus with two interior marked points. The central and right examples are mixed-angulations of the disk. The central one has singular points with degrees $1,2$, and $4$ which are the centers of $1$-, $2$-, and $4$-gons, respectively. The right one includes boundary singular points and infinitely many boundary marked points arising from $\infty$-gons.}
\label{fig:S-graphs}
\end{figure}

\Cref{sec_gbga} is dedicated to defining RGB algebras, and describing their Koszul duals. We also show in \Cref{prop:nCY}, that the RGB algebra is $n$-Calabi--Yau if the underlying surface has no boundary singular points and satisfies a certain orientability assumption. We also describe a counterexample in \Cref{rem:noCY}, showing that this orientability assumption cannot be dropped. 

We show that the derived category of RGB algebras arise and can be studied using two very different approaches. The first approach uses $A_\infty$-deformations of Fukaya categories of surfaces and the second approach uses categorified pervese sheaves.

\subsection{RGB algebras from Fukaya \texorpdfstring{$A_\infty$}{A-infinity}-categories of surfaces}\label{introsec:FukayaAinfty}

In \Cref{sec_deffuk} we show that RGB algebras naturally arise from certain 1-parameter deformations of (partially wrapped) Fukaya categories of surfaces. 
This approach, which relies on explicit curved $A_\infty$-structures, also works in the case of closed surfaces (no marked points or boundary). 
As a byproduct, we obtain a proof of the correspondence between quadratic differentials and stability conditions for derived categories of RGB algebras.

The construction of the $A_\infty$-category is based on~\cite{Hai21}, but generalizes the setup in three major ways:
\begin{enumerate}
    \item Consider $n$CY categories for any $n\in\mathbb Z$ instead of just $n=3$,
    \item allow non-closed surfaces, and
    \item allow marked points on the boundary.
\end{enumerate}
These correspond, roughly, to including quadratic differentials with 1) higher order zeros, 2) higher order poles and 3) exponential singularities (as in~\cite{HKK17}), respectively.

The starting point of the construction is the Fukaya $A_\infty$-category $\mathcal F(S)=\mathcal F(S,M,\nu)$ of a surface $S$, possibly with boundary, marked points $M\subset S$, and grading structure (line field) $\nu$ on $S$.
For a choice of subset $M'\subset M$ of interior marked points and integer $n\in \ZZ$ which is a positive multiple of the indices of $\nu$ at the points in $M'$, one defines a curved, $\mathbb Z\times\mathbb Z/2$-graded deformation of $\mathcal F(S)$ over $\mathbf k[[t]]$, $|t|=2-n$. (In particular, we do not need to impose the condition $n\geq 2$ as in the perverse schober construction.)
Roughly, the deformation is obtained by counting disks with punctures which map to the punctures in $M'$.
A triangulated $A_\infty$-category $\mathcal C(S,n)=\mathcal C(S,M,\nu,M',n)$ is defined as torsion modules over the deformation (i.e.~objects living on the total space of the deformation but supported on the central fiber).
The categories $\mathcal C(S,n)$ are often $n$-Calabi--Yau or at least (right) relative Calabi--Yau.
Base change from $\mathbf k[[t]]$ to $\mathbf k$ defines a pair of adjoint functors between $\mathcal F(S)$ and $\mathcal C(S,n)$.
A general result from~\cite{Hai21} allows us to transfer stability conditions from a certain full subcategory $\mathcal F_{\mathrm{len}}(S)$ of $\mathcal F(S)$ to the corresponding full subcategory $\mathcal C_{\mathrm{len}}(S,n)$ and establish the following.

\begin{theorem}[\Cref{thm_defstab}]\label{introthm:transferstab}
    Let $S$, $M$, $\nu$, $M'$, $n$ be as above and $n\geq 3$. If $\mathcal M(S,M,\nu)$ denotes the moduli space of quadratic differentials attached to $(S,M,\nu)$ as in~\cite{HKK17}, see \Cref{subsec_transfer}, then there is a canonical map
    \[
    \mathcal M(S,M,\nu)\longrightarrow \mathrm{Stab}\left(\mathcal C_{\mathrm{len}}(S,n)\right)
    \]
    which is a biholomorphism onto a union of connected components.
\end{theorem}

Let us emphasize that any quadratic differential on a compact Riemann surface, possibly with zeros, poles, and exponential singularities, appears in one of the moduli spaces $\mathcal M(S,M,\nu)$.
Moreover, the categories $\mathcal C_{\mathrm{len}}(S,n)$ are often better behaved than $\mathcal F_{\mathrm{len}}(S)$ in the sense that they are proper and/or (relative) Calabi--Yau.

The relation with RGB algebras is described by the following result.

\begin{theorem}[\Cref{thm:nonformalgen}] \label{introthm:generator}
    Suppose $M\neq\emptyset$, $M'=M\cap\mathrm{int}(S)$, and $n\geq 1$.
    Choose an S-graph, $\Sgh$, on $(S,M,\nu)$.
    Then the union of the edges of $\Sgh$ defines a generator $G$ of a full subcategory $\mathcal C_{\mathrm{core}}(S,n)$ of $\mathcal C_{\mathrm{len}}(S,n)$ and the endomorphism algebra of $G$ is quasi-equivalent to the RGB algebra $A(\Sgh,n)$ (see Section~\ref{sec_gbga}).
\end{theorem}

We note that $\mathcal C_{\mathrm{core}}(S,n)$ is often equal to $\mathcal C_{\mathrm{len}}(S,n)$ and has the same space of stability conditions. 

\subsection{RGB algebras and perverse schobers}

Perverse schobers are a notion of categorified perverse sheaf proposed by Kapranov--Schechtman \cite{KS14}, in which vector spaces are replaced by enhanced triangulated categories (we will use stable $\infty$-categories for this). The notion of a perverse schober remains conjectural on general stratified spaces, but on a marked surfaces, we can employ the notion of a perverse schober surface parametrized by a ribbon graph, described in \cite{Chr22}. Such a perverse schober is encoded in terms of a constructible sheaf of stable $\infty$-categories on the ribbon graph $\rgraph$ satisfying local conditions. Concretely, this amounts to a functor $\on{Exit}(\rgraph)\to \on{St}$ to the $\infty$-category of stable $\infty$-categories $\on{St}$, where $\on{Exit}(\rgraph)$ is the exit path $\infty$-category of $\rgraph$. 

In the prequel article \cite{CHQ23}, we studied the tilting theory of the stable $\infty$-categories of global sections of perverse schobers. Under certain local conditions, we matched finite length hearts of $t$-structures with mixed-angulations of the surface. As a consequence, we obtained an embedding of a space of framed quadratic differential into the space of Bridgeland stability conditions. In \Cref{subsec:BGAschober} of this article, we explore how the derived $\infty$-categories of RGB algebras relate with perverse schobers and explain how the results of \cite{CHQ23} apply.  

\subsubsection{Global sections and Koszul duality}\def\psG{\hF_{\Sgh,G}}

Given a perverse schober on a surface, its $\infty$-category of global sections can be considered as the topological Fukaya category of the surface with coefficients in the perverse schober. Topological Fukaya categories themselves, which are equivalent to the derived categories of graded gentle algebras, arise as the global sections of the simplest examples of perverse schobers \cite{DK18,DK15}. More elaborate examples are given by the derived $\infty$-categories of the (higher) relative Ginzburg dg-algebras associated with $n$-angulated surfaces. These arise as the global section of perverse schobers parametrized by the dual $n$-valent ribbon graphs of the $n$-angulations \cite{Chr22,Chr21b}. The perverse schober assigns to a vertex $v$ of the dual graph the derived $\infty$-category of the relative Ginzburg algebra of an $n$-gon, denoted in the following by $G_n$. This dg-algebra has a $\mathbb{Z}/n$ symmetry corresponding to the rotation of the $n$-gon. Given $m\geq 1$ dividing $n$, we can pass to the $\ZZ/\frac{n}{m}$-orbit dg-algebra. In characteristic not divided by $\frac{n}{m}$, we show in \Cref{prop:quotschob} that its derived $\infty$-category gives the value of a new perverse schober on an $m$-gon, instead of the $n$-gon. 

For instance, for $n=3$, the relative Ginzburg algebra of a $3$-gon $G_{3}$ has the underlying graded quiver
\[
\begin{tikzcd}
                                               & 2 \arrow[rd, "b", bend left] \arrow[ld, "a^*"] &                                                \\
1 \arrow[ru, "a", bend left] \arrow[rr, "c^*"] &                                                & 3 \arrow[ll, "c", bend left] \arrow[lu, "b^*"]
\end{tikzcd}
\]
with $|a|=|b|=|c|=0$, $|a^*|=|b^*|=|c^*|=-1$, with the differentials determined by the potential $W=cba$. Passing to the $\ZZ/3$-orbit, we obtain the dg-algebra $G_{1}$ with underlying graded quiver 
\[
\begin{tikzcd}
1 \arrow["a"', loop, distance=2.5em, out=125, in=55] \arrow["a^*", loop, distance=5em, out=140, in=40]
\end{tikzcd}
\]
with potential $W=a^3$.

Given a mixed-angulated surface $\sow$ together with a choice of an S-graph $\Sgh$ and compatible $n\geq 3$, we define the dg-algebra $\GSn$ by gluing together cofibrant versions of the cyclic quotients of the relative Ginzburg algebra of an $n$-gon $G_{n}$ along the S-graph. We further associate novel types of local dg-algebras with vertices lying in $\infty$-gons, which are a mixture of an $A_l$-quiver and a Ginzburg algebra. If $n=3$ and $\Sgh$ has only trivalent and $1$-valent vertices, the corresponding Jacobian algebra $H^0(G(\Sgh,3))$ is gentle and previously appeared in \cite{lfm22}. Supposing that $\on{char}(k)$ is not divided by $\frac{n}{m}$ with $m$ the degree of any singular point, we show the following:

\begin{theorem}[\Cref{thm:BGAschober}]\label{introthm:gluing}
The derived $\infty$-category $\D(\GSn)$ arises as the $\infty$-category of global sections of a $\eS$-parametrized perverse schober $\psG$. 

The value of $\psG$ at an $m$-valent vertex of degree $m$ of the S-graph $\Sgh$ is given by the unbounded derived $\infty$-category of the $\ZZ/\frac{n}{m}$-orbit dg-algebra $G_{n,m}$ of the relative Ginzburg algebra $G_{n}$ of the $n$-gon.   
\end{theorem}

We expect that \Cref{introthm:gluing} remains true as stated in arbitrary characteristic, and that this can be proved using group quotients of the derived $\infty$-categories, instead of quotient dg-algebras. 

A central observation is that there is a quasi-isomorphism $\GSn\simeq A(\Sgh,n)^!$ with the Koszul dual of the relative graded Brauer graph algebra $A(\Sgh,n)$ described in \Cref{subsec_koszuldual}. Hence, there is an equivalence of stable $\infty$-categories $\per(A(\Sgh,n))\simeq \D^{\on{nil}}(\GSn)$, where $\D^{\on{nil}}(\GSn)\subset \D(\GSn)$ denotes the nilpotent derived $\infty$-category, which is given by the stable subcategory generated by the simple modules associated with the vertices of the underlying quiver. The nilpotent derived $\infty$-category $\D^{\on{nil}}(\GSn)$ is further a full subcategory of the finite derived $\infty$-category $\D^{\on{fin}}(\GSn)$ consisting of $\GSn$-modules whose underlying $k$-module is perfect. For a discussion of the relation between $\per(A(\Sgh,n))\simeq \D^{\on{nil}}(\GSn)$ and $\D^{\on{fin}}(\GSn)$, as well as a sufficient condition for them to coincide, see \Cref{lem:Dfin=Dnil} and the discussion afterwards. 
 
A global section of the perverse schober $\psG$ lies in $\D^{\on{fin}}(\GSn)$ if and only if all its restrictions to the local values of $\psG$ lie in the local finite derived categories. These are described by the perfect derived categories of relative graded Brauer graph algebras by \Cref{lem:Dfin=Dnil}. \Cref{introthm:gluing} thus implies:

\begin{corollary}\label{cor:intro}
Consider the $\eS$-parametrized perverse subschober $\psG^{\on{fin}}\subset \psG$ that assigns to each vertex $v$ of the S-graph the perfect derived $\infty$-category $\per(A)$ of the RGB algebra $A$ whose Koszul dual $A^!$ is the generalized relative Ginzburg dg-algebra with $\D(A^!)\simeq \psG(v)$. 

Then the $\infty$-category of global sections\footnote{By which we mean here the limit of $\psG^{\on{fin}}$ in the $\infty$-category $\on{St}$ of stable $\infty$-categories.} of $\psG^{\on{fin}}$ is equivalent to $\D^{\on{fin}}(\GSn)$, into which $\per(A(\Sgh,n))$ embeds fully faithfully.
\end{corollary}

We also show that $\GSn$ arises as a global group quotient of a relative Ginzburg algebra associated with an $n$-angulated surface, see \Cref{pp:covering}. More precisely, we consider a slight generalization of such a relative Ginzburg algebra, since we allow vertices of the S-graph $\Sgh$ of degree $\infty$, which can be of arbitrary valency.

\subsubsection{Simple minded collections and stability conditions}

The results of \cite{CHQ23} allow the construction of a simple minded collection in the $\infty$-category of global sections of a perverse schober, given as input certain local information about the perverse schober referred to as an arc system kit. In fact, a simple minded collection will be produced for every S-graph $\Sgh$ of the weighted marked surface, we denote this collection by $\A_\Sgh$. The objects of $\A_\Sgh$ are in bijection with the edges of the S-graph. The collection $\A_\Sgh$ forms the simples in the heart of a bounded $t$-structure on the stable $\infty$-category $\mathcal{C}(\Sgh,\psG)$ generated by $\A_\Sgh$. Further, the stable $\infty$-category $\mathcal{C}(\Sgh,\psG)$ does not change under flips of the S-graph. As shown in \cite{CHQ23}, the simple tilting at a simple object in $\A_\Sgh$ yields the simple minded collection associated with the corresponding flipped S-graph of $\Sgh$.  

We show in \Cref{prop:generatedstablecategory} that $\psG$ admits an arc system kit, such that $\mathcal{C}(\Sgh,\psG)\simeq \per(A(\Sgh,n))$. The main result of \cite{CHQ23} then applies to describe the space of Bridgeland stability conditions on $\per(A(\Sgh,n))$, giving an alternative proof of \Cref{introthm:transferstab} for non-closed surfaces and in good characteristic.

\begin{theorem}[{\cite[Thm.~5.4]{CHQ23}}]\label{introthm:stab}
There is a map
\[ \label{eq:app*}
    \FQuad{S}{\sow} \to \Stab(\per(A(\Sgh,n)))
\]
from the space of framed quadratic differentials on $\sow$ to the space of Bridgeland stability conditions, which is a biholomorphism onto a union of connected components.
\end{theorem}

We highlight in \Cref{introex:Ginzburg} the case $n=3$, which is particularly interesting from the perspective of Donaldson--Thomas theory.

\begin{example}\label{introex:Ginzburg}
Let $n=3$ and let $\sow$ be a graded marked surface whose interior singular points are all of degree $3$. The group actions considered above are all chosen trivial in this case. If we are given an S-graph $\Sgh$ of $\sow$, whose interior vertices are $3$-valent and whose boundary vertices are $1$-valent, the corresponding dg-algebra $\GSn$ exactly recovers the relative Ginzburg algebra considered in \cite{Chr22,Chr21b}. If $\Sgh$ also has interior vertices of valency $<3$, $\GSn$ is an intermediate version lying between the fully relative and non-relative Ginzburg algebras. If $\Sgh$ has boundary vertices of valency $>1$, the corresponding algebra $\GSn$ is of a more general type than typical relative/non-relative Ginzburg algebras. 

In the above setting the space of stability conditions on $\per(A(\Sgh,n))$ is described by quadratic differentials with poles of order $2$ and above, simple zeros and exponential singularities. The derived $\infty$-category $\mathcal{D}(\GSn)$ is relative left $3$-Calabi--Yau in the sense of \cite{BD19}, as follows from \cite{Chr23}

If we further include singular points of degree $1$ in $\sow$, we obtain a more general dg-algebra $\GSn$, which we expect to still be relative $3$-Calabi--Yau. The corresponding quadratic differentials now also have simple poles. 

Suppose we instead include singular points of degrees $1$ and $3$, but no boundary singular points, which means we consider mixed-angulations with monogons and triangles only. Then the RGB algebra $A(\Sgh,n)$ is $3$-Calabi--Yau (without the adjective relative), see \Cref{prop:nCY}. In this case, the defining quiver with potential of $\GSn$, as well as the relation between flips and mutations, previously appeared in \cite{lfm22}. The space of stability conditions is described by quadratic differentials with poles of order $n\geq 1$ and simple zeros. Building on the results of this article and its prequel \cite{CHQ23}, the corresponding Donaldson--Thomas invariants have been computed in \cite{KW24}. 
\end{example}

\subsubsection{Open problems}
Finally, we comment on possible directions of further investigation concerning perverse schobers and RGB algebras. 

While it is clear how to define RGB algebras $A(\Sgh,n)$ for $n=0$ and $n<0$, we restrict in the construction of the corresponding perverse schober to $n>0$. This is because we construct the perverse schobers from cyclic quotients of relative Ginzburg algebras of $n$-gons, which have previously been studied only in positive Calabi--Yau dimension.  
It would be interesting to generalize these construction to the case of arbitrary $n$.

The stable module category of a Brauer graph algebra $A$ is equivalent to its singularity category $\D^{\on{fin}}(A)/\per(A)$, which is in turn equivalent to the cosingularity category $\per(A^!)/\D^{\on{fin}}(A^!)$ of the Koszul dual dg-algebra $A^!$. We expect this relation between the singularity category and the cosingularity category to extend to the RGB algebra $A(\Sgh,n)$ and its Koszul dual $\GSn$ (for any $n\in \mathbb{Z}$). When $A(\Sgh,n)$ is $n$-Calabi--Yau, see \Cref{prop:nCY}, this follows for instance from \cite{GS20}.

It would be very interesting to determine when the passage to the cosingularity category of $\GSn$, $n\in \mathbb{Z}$, commutes with the gluing in terms of the perverse schober. In the case $n=3$, this commutativity is shown in \cite{Chr22b} for the relative Ginzburg algebra of a marked surface with a marked point (i.e.~singular point from the perspective of this paper) on each boundary circle. The corresponding description as the global sections of a perverse schober could be used for instance to give descriptions of the objects and morphisms in this singularity category in terms of curves in the surface. 

The cosingularity category of the non-relative Ginzburg algebra further arises as an exact localization of the cosingularity category of the relative Ginzburg algebra \cite{Chr22b}. The exact structure on the cosingularity category of the relative Ginzburg algebra is induced by the functor to the boundary. This raises the question whether the singularity category of a graded or ungraded non-relative Brauer graph algebra also arises as an exact localization of the singularity category of a relative graded or ungraded Brauer graph algebra. In the case $n=3$, it would also be interesting to investigate how the arising cosingularity category categorifies the generalized cluster algebras considered in \cite{lfm22}. 

We expect that the derived $\infty$-category of $A(\Sgh,n)$ carries under certain orientation assumptions a relative right $n$-Calabi--Yau structure in the sense of Brav--Dyckerhoff \cite{BD19}, this generalizing \Cref{prop:nCY} corresponding to the case where there are no $\infty$-gons. This relative right $n$-Calabi--Yau structure should arise from restricting a relative left $n$-Calabi--Yau on $\GSn$, see also \Cref{rem:nCY}.

\subsection{Notation}
\begin{itemize}
\item Weighted marked surface $\sow=(\bf S,\M,\W,\wt,\nu)$
    with $\M$ the marked points, $\W$ the singular points, $\wt\colon\W\rightarrow \mathbb{N}_{\geq -1}\cup\{\infty\}$ the weight function, and $\nu$ the line field.
\item Mixed-angulation $\AS$ of $\sow$, with dual graph an S-graph $\Sgh=\dAS$.
\item Forward flip $\AS^\sharp_\gamma$ of a mixed-angulation $\AS$ at an arc $\gamma\in\AS$.
\item Exchange graph $\EG(\sow)$ and exchange graph of S-graphs $\EG_S(\sow)$ of the weighted marked surface~$\sow$.
\item Moduli space of framed quadratic differentials $\FQuad{}{\sow}$.
\item The category $\glsec(\rgraph,\hF)$ of global sections of the perverse schobers $\hF$ parameterized by a ribbon graph $\rgraph$.
\item Extended ribbon graph $\eS$ of an S-graph $\Sgh$.
\item Collection of object $\ASG\subset \Gamma(\eS,\hF)$ corresponding to the edges of the S-graph, generating the stable subcategory $\CS$ with a canonical heart $\CSh$.
\end{itemize}

Note that our convention of forward flip (moving endpoints clockwise) and morphism direction (counter-clockwise) are the inverse of the ones in \cite{Qiu16,KQ2,BMQS}.

\subsection{Acknowledgements}
We thank Wassilij Gnedin, Sebastian Opper, Sibylle Schroll, and Alexandra Zvonareva for helpful conversations about Brauer graph algebras. We further thank Patrick Le Meur for explaining his work~\cite{lemeur20}.

\section{Relative graded Brauer graph algebras}
\label{sec_gbga}

In this section we give a construction of dg-algebras from S-graphs which is Koszul dual to the one in Subsection~\ref{subsec:BGAschober}.
If the S-graph has no boundary vertices, then this dg-algebra is just a graded algebra, in fact a graded enhancement of a \textit{Brauer graph algebra} in sense of~\cite{donovan_freislich}. In the presence of boundary vertices they are a relative and graded variant of Brauer graph algebras.
These dg-algebras will reappear in Subsection~\ref{subsec_sgraphs} and thus provide the link between the two approaches --- perverse schobers and $A_\infty$-categories.

\subsection{Mixed-angulated surfaces and S-graphs}
\label{subsec:WMS}

In this section we review the central notion of a mixed-angulation on a weighted marked surface, as introduced in \cite{CHQ23}, and the dual notion of an S-graph. We assume the reader is familiar with the notion of grading of surfaces and curves, see for example~\cite[Section 2.1]{CHQ23}.

The following types of decorated surfaces provide a home for mixed-angulations and mutations between them.

\begin{definition}\label{def:weightedmarkedsurf}
A \emph{weighted marked surface} $\sow=(\bf S,\M,\W,\wt,\nu)$ is given by
\begin{itemize}
\item a connected compact oriented surface $\bf S$, possibly with boundary,
\item a non-empty subset $\M\subset {\bf S}$ of \textit{marked points} (vertices),
\item a finite subset $\W\subset {\bf S}$ of \textit{singular points} (centers of polygons),
\item a \textit{weight function} $\wt\colon \W\to \ZZ_{\geq -1}\cup \{\infty\}$,
\item a grading structure (foliation) $\nu$ on $\bf S\setminus (\M\cup \W)$.
\end{itemize}
satisfying the following:
\begin{enumerate}
\item $\wt(x)=\infty \iff x\in\partial \bf S$,
\item $\M\cap \W=\emptyset$,
\item $|\mathrm{int}(\bf S)\cap \M|<\infty$,
\item $\M$ intersects each component of $\partial \bf S$,
\item $\M$ is discrete in $\bf S\setminus \W$, and any non-compact component of $\partial \bf S\setminus \W$ contains countably infinitely many points of $\M$,
\item the index of any $x\in\W\cap \mathrm{int}(\bf S)$ with respect to $\nu$, $\mathrm{ind}_\nu(x)$, is equal to $d(x)\coloneqq\wt(x)+2$, the \textit{degree} of $x\in\W$.
\end{enumerate}
An \emph{isomorphism} of weighted marked surfaces $\sow\to\sow'$ is a isomorphism of graded surfaces $(f,h):({\bf S},\nu)\to ({\bf S}',\nu')$ with $f(\M)=\M'$, $f(\Delta)=\Delta'$, and $\wt'\circ f=\wt$.
\end{definition}
\begin{definition}\label{def:arc}
A \emph{compact arc} in $\sow$ is an immersed curve $\gamma\colon [0,1]\to\surf$ such that $\gamma(\{0,1\})\subseteq\W$, $\gamma((0,1))\cap (\W\cup \M)=\emptyset$, and $\gamma|_{(0,1)}$ is embedded.
If $\gamma(0)=\gamma(1)$, then $\gamma$ should not be homotopic in ${\bf S}\setminus\M$, relative its endpoints, to a constant loop.
A \emph{non-compact arc} is defined like a compact arc but with the roles of  $\M$ and $\W$ reversed, in particular the endpoints lie in $\M$ instead of $\W$.
A \emph{boundary arc} is a non-compact arc which cuts out a bigon in ${\bf S}$ whose other edge is a part of $\partial\bf S$ containing exactly one point of $\W$ and containing no points of $\M\cup \W$ in its interior.
A \emph{closed curve} is an immersed curve $\gamma\colon S^1\to\mathrm{int}(\bf S)\setminus(\M\cup \W)$.
\end{definition}

The following notion generalizes the well known \textit{ideal triangulations}.

\begin{definition}\label{def:mixedangulation}
    A \emph{mixed-angulation} of a weighted marked surface, $\sow$, is a finite set, $\mathbb A$, of non-compact graded arcs in ${\bf S}$, the \textit{internal edges}, intersecting only in endpoints and cutting $\bf S$ into polygons. These \textit{$\AS$-polygons} have vertices in $\M$, edges which are arcs in $\mathbb A$ and/or boundary arcs (aka \textit{boundary edges}), and contain exactly one singular point $x\in\W$ where $d(x)$ is the number of edges of the polygon.
    Moreover, we impose the following constraint on the grading of the arcs: $i(X,Y)=0$ if $X$,$Y$ are two consecutive edges of a polygon in clockwise order. Here we have also chosen a (uniquely determined by this condition) grading on the boundary arcs.
\end{definition}

We recall the definition of an \textit{S-graph} from~\cite{HKK17}. 
These generalize ribbon graphs (with boundaries) and are dual to mixed-angulations,
where we regard singular points with infinite weight as the center of the polygons with infinity many edges.
In the following, we allow graphs with loops and multiple edges.

\begin{definition}\label{def:sgraph}
    An \emph{S-graph} is a graph, $\Sgh$, with the following additional structure: 1) a partition of the set of vertices into \textit{internal vertices} and \textit{boundary vertices}, 2) for each internal (resp. boundary) vertex $v$, a cyclic (resp. total) order on the set of halfedges meeting $v$ and 3) for each pair $h_1,h_2$ of successive halfedges a number $d(h_1,h_2)\in\ZZ_{>0}$.
\end{definition}

Given an S-graph, it is convenient to define $d(h_1,h_2)$ for more general pairs of halfedges attached to the same vertex by the rule that $d(h_1,h_3)=d(h_1,h_2)+d(h_2,h_3)$ if the sequence $(h_1,h_2,h_3)$ is compatible with the cyclic/total order at $v$.

\begin{definition}
    A mixed-angulated surface $(\sow,\mathbb A)$ determines a \textit{dual S-graph} $\Sgh=\dAS$, embedded in $\sow$ in the following way.
    The vertices of $\Sgh$ are $\W$ and the edges of $\Sgh$ are dual to the internal edges of $\mathbb A$, see Figure~\ref{fig:mixangulation}.
    The cyclic/total order on the set of halfedges meeting a given vertex $x\in \W$ is induced from the counter-clockwise order around $x$ in $\sow$. Additionally, $d(h_1,h_2)-1$ is the number of boundary edges between the interior edges corresponding to $h_1$ and $h_2$ as one goes counter-clockwise around the polygon.
    The grading on the arcs of $\mathbb A$ determines a grading on the edges of $\Sgh$ by the requirement that $i(X,Y)=0$ if $X\in\mathbb A$ and $Y$ is the dual edge.
\end{definition}

\begin{figure}[ht]\centering
\makebox[\textwidth][c]{
 \includegraphics[width=7cm]{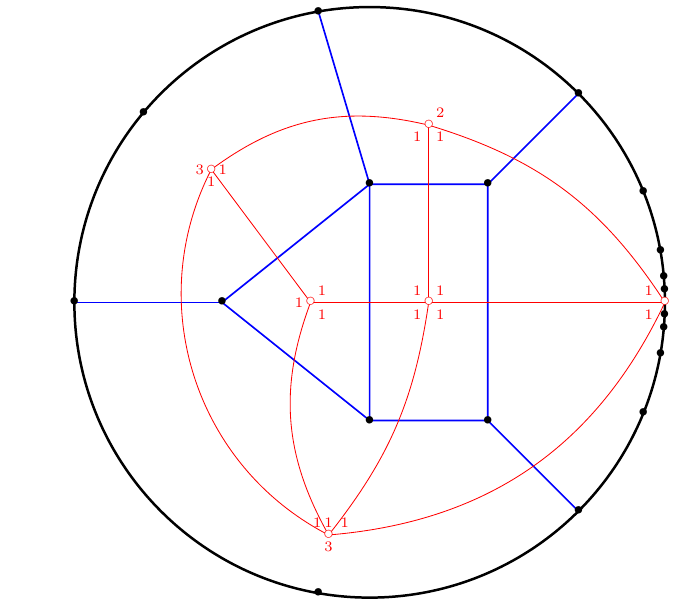}
}
\caption{Example of surface with mixed-angulation (blue) and dual S-graph (red) with its numbers $d(h_1,h_2)$.}
\label{fig:mixangulation}
\end{figure}

The important notion of a \textit{flip} is defined as follows for mixed-angulations.

\begin{definition}\label{def:flip}
Given a mixed-angulation $\AS$ and an arc $\gamma\in\AS$ we define the \emph{forward flip} $\AS^\sharp_\gamma$ to be the mixed-angulation for which $\gamma$ has been replaced by the arc $\gamma^\sharp$ obtained by rotating $\gamma$ clockwise so that each of its endpoints have moved along an adjacent edge of an $\AS$-polygon, see Figures~\ref{fig:flip1} and~\ref{fig:flip2}.
The inverse construction is the \emph{backward flip}, denoted by $\AS\mapsto\AS^\flat_\gamma$.

The grading of the arcs in the forward/backward flip is inherited from the one in the original mixed-angulation. We refer to \cite[Sec.~2.1]{CHQ23} for more information.
\end{definition}

\begin{figure}[ht]\centering
\makebox[\textwidth][c]{
 \includegraphics[width=12cm]{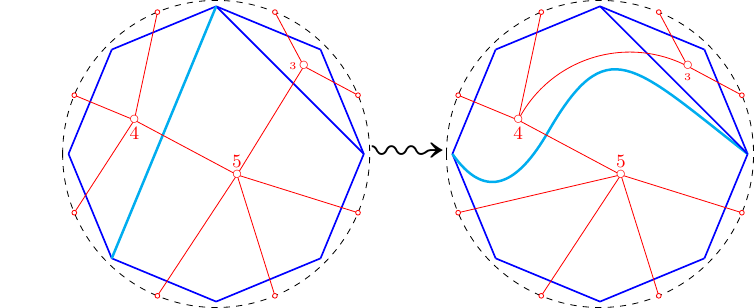}\qquad}
\caption{The forward flip at a usual arc.}
\label{fig:flip1}
\end{figure}

\begin{figure}[ht]\centering
\makebox[\textwidth][c]{
 \includegraphics[width=12cm]{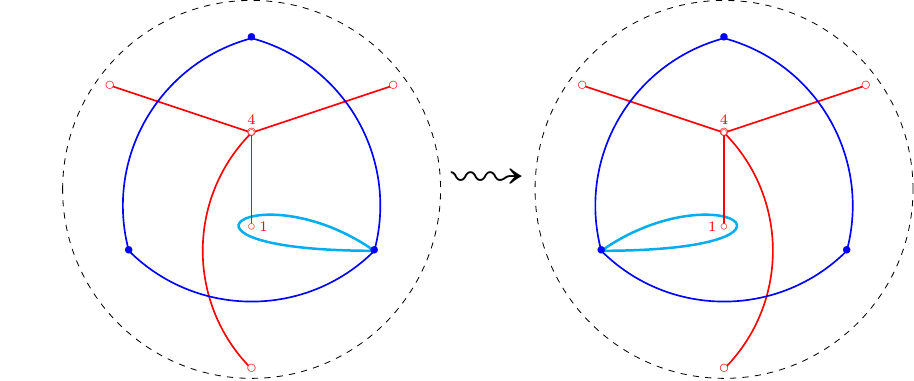}\qquad}
\caption{The forward flip at a monogon arc.}
\label{fig:flip2}
\end{figure}

\subsection{From S-graphs to RGB algebras}

Fix a coefficient field $\mathbf k$.
In the following, an integer $n$ is said to be \textit{compatible} with an S-graph $\Sgh$ if $n$ is a positive multiple of the degrees of all internal vertices of $\Sgh$. If there are no internal vertices, then $n$ is arbitrary.

\begin{definition}\label{def_gbga}
The \emph{relative graded Brauer graph algebra (RGB algebra)} of an S-graph $\Sgh$ and a compatible integer $n$ is the dg-algebra $A(\Sgh,n)$ given by the following graded quiver with relations and differential:
\begin{enumerate}
    \item Vertices are the edges of $\Sgh$.
    \item Arrows:
    \begin{itemize}
        \item for each pair $(i,i+1)$ of successive halfedges (``corners'') of $\Sgh$, there is an arrow $a_i$ from $i$ to $i+1$ of degree $|a_i|=d(i,i+1)\in\ZZ$. The arrows thus go in the counterclockwise direction.
        \item for each halfedge $i$ attached to a boundary vertex, there is a loop $\tau_i$ at $i$ of degree $|\tau_i|=n-1$.
    \end{itemize}
    \item Relations:
    \begin{itemize}
        \item $a_{j}a_i=0$ if $i+1\neq j$ are halfedges belonging to the same edge.
        \item for each edge $\{i,j\}$ attached to internal vertices at both ends: $c_i=(-1)^{n-1}c_j$, where
        \[
        c_i\coloneqq \left(a_{i-1}a_{i-2}\cdots a_{i+1}a_{i}\right)^{\frac{n}{m}}
        \]
        is the cycle going $n/m$ times around the vertex $v$ to which $i$ is attached, which starts and ends at $i$, and $m\coloneqq\deg(v)$ (see Figure~\ref{fig:gbga_cycle}),
        \item $\tau_i^2=0$
        \item for each pair $(i,i+1)$ of successive halfedges attached to a boundary vertex: $a_i\tau_i=(-1)^{|a_i|}\tau_{i+1}a_i$
        \item for each edge $\{i,j\}$ attached only to boundary vertices: $\tau_i=(-1)^n\tau_j$.
    \end{itemize}
    \item Differential: for each edge $\{i,j\}$ attached to an internal vertex along $i$ and a boundary vertex along $j$: $d(\tau_j)=(-1)^nc_i$. The differentials of the other generators vanish.
\end{enumerate}
\end{definition}

\begin{figure}[ht]
    \centering
    \begin{tikzpicture}[>=Stealth]
        \foreach \n in {0,1,2}
        {
            \coordinate (v) at (360/3*\n:3);
            \coordinate (w) at (360/3*\n+360/3:3);
            \draw[thick,red] (0,0) -- (120*\n:3);
            \draw[thick,red,dotted] (120*\n:3) -- (120*\n:3.6);
            \draw[red] (120*\n-5:2.5) node {$\n$};
            \draw[thick,->] (120*\n:.6) arc [start angle=360/3*\n,end angle=360/3*\n+120,radius=.6];
            \draw (120*\n+60:.8) node {$a_{\n}$};
        }
        \draw[thick,->] (1,0) to[out=90,in=0] (0,1.1) to[out=180,in=90] (-1.2,0) to[out=-90,in=180] (0,-1.3) to[out=0,in=-90] (1.4,0) to[out=90,in=0] (0,1.5) to[out=180,in=90] (-1.6,0) to[out=-90,in=180] (0,-1.7) to[out=0,in=-90] (1.8,0);
        \draw (-60:2) node {$c_0$};
        \draw[red] (0,0) \ww;
    \end{tikzpicture}
    \caption{Schematic view of the cycle $c_0$ in the case of a trivalent vertex of the S-graph with $n/\deg(v)=2$.}
    \label{fig:gbga_cycle}
\end{figure}

\begin{remark}
    The definition of S-graph includes the condition $d(i,i+1)\geq 1$ for any pair of successive halfedges $(i,i+1)$.
    Thus, the degree of a vertex (the sum over $d(i,i+1)$'s for all $i$'s belonging to the vertex)
    is always positive.
    However, the constructions in this section work just as well for any $d(i,i+1)\in\ZZ$
    (in the case of generalized ribbon graphs with boundaries).
    
    The case $n=0$ would however need an extra modification, the exponent $n/m$ in the definition of $c_i$ in \Cref{def_gbga} must be replaced by a positive integer assigned to the vertex.
\end{remark}

\begin{remark}
    Our choice of signs in \Cref{def_gbga} is motivated by the connection with Fukaya categories of surfaces, see Section~\ref{sec_deffuk} below.
    In the case $n=0$, in particular for the special case of ungraded Brauer graph algebras, our signs are opposite to the usual ones, but have been considered previously~\cite{gss14,gnedin19}.
\end{remark}

\begin{figure}[ht]
    \centering
    \includegraphics[width=10cm]{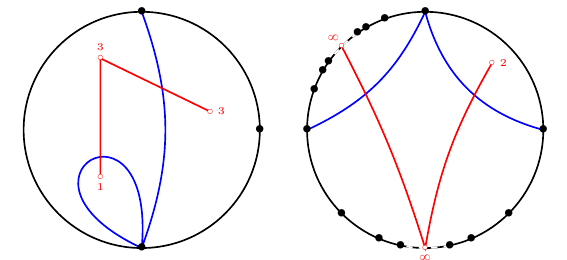}
    \caption{Two examples of decorated marked surfaces with mixed-angulation.}
    \label{fig:2examples}
\end{figure}

\begin{example}
Consider the mixed-angulation of the disk shown on the left in \Cref{fig:2examples}. There are two triangles and a monogon, so we can choose $n=3m$ to be any positive multiple of $3$. 
Since there are no $\infty$-gons, the RGB algebra $A(\Sgh,n)$ is just a graded algebra.
It is described by the quiver
\[
\begin{tikzcd}[column sep=huge, row sep=huge]
    1 \arrow[loop left,"a_{11}"]\arrow[r,swap,bend right=10,"a_{21}"] & 2 \arrow[loop right,"a_{22}"] \arrow[l,swap,bend right=10,"a_{12}"]
\end{tikzcd}
\]
where the arrows have degrees $|a_{11}|=|a_{21}|=1$, $|a_{12}|=2$, $|a_{22}|=3$ and the following relations hold:
\begin{gather*}
a_{11}a_{12}=0,\qquad a_{21}a_{11}=0,\qquad a_{22}a_{21}=0,\qquad a_{12}a_{22}=0, \\
a_{11}^{3m}=\left(a_{12}a_{21}\right)^m,\qquad a_{22}^m=\left(a_{21}a_{12}\right)^m.
\end{gather*}
A basis of $A(\Sgh,3)$ is given by the following elements:
{\renewcommand{\arraystretch}{1.4}\setlength{\tabcolsep}{8pt}
\begin{center}
\begin{tabular}{rcccc}
    degree & 0 & 1 & 2 & 3 \\ \hline
     basis elements & $e_1$, $e_2$ & $a_{11}$, $a_{21}$ & $a_{11}^2$, $a_{12}$ & $a_{11}^3$, $a_{22}$ 
\end{tabular}
\end{center}
}
As we will see below, $A(\Sgh,n)$ is $n$CY in the appropriate sense.
\end{example}

\begin{example}
Consider the mixed-angulation of the disk shown on the right in \Cref{fig:2examples}.
There are two $\infty$-gons and a bigon, so we can choose $n=2m$ to be any positive even number. 
The RGB algebra $A(\Sgh,n)$ is described by the quiver
\begin{center}
\begin{tikzcd}[column sep=huge, row sep=huge]
    1 \arrow[loop above,"\tau_{1}"] & 2 \arrow[loop right,"a_{22}"] \arrow[l,swap,"a_{12}"] \arrow[loop above,"\tau_{2}"]
\end{tikzcd}    
\end{center}
where the arrows have degrees $|a_{12}|=1$, $|a_{22}|=2$, $|\tau_1|=|\tau_2|=n-1$ and the following relations hold:
\[
a_{12}a_{22}=0,\qquad \tau_1^2=\tau_2^2=0,\qquad \tau_1a_{12}=-a_{12}\tau_2.
\]
The differential is given on generators by
\[
da_{12}=da_{22}=d\tau_1=0,\qquad d\tau_2=a_{22}^m.
\]
A basis of $A(\Sgh,2)$ is given by the following elements:
{\renewcommand{\arraystretch}{1.4}\setlength{\tabcolsep}{8pt}
\begin{center}
\begin{tabular}{rcccc}
    degree & 0 & 1 & 2 \\ \hline
     basis elements & $e_1$, $e_2$ & $a_{12}$, $\tau_1$, $\tau_2$ & $a_{22}$, $\tau_1a_{12}$
\end{tabular}
\end{center}
}
\end{example}

Suppose all vertices of $\Sgh$ are interior (there is a \textit{cyclic} order on the set of halfedges meeting any vertex),
then there are no $\tau_i$'s and the differential of $A(\Sgh, n)$ vanishes, so that $A(\Sgh, n)$ defines a graded algebra.
Always in the case of odd $n$ and under an orientability condition in the case of even $n$, this algebra turns out to be $n$-Calabi--Yau in the following sense, see \Cref{prop:nCY}.

\begin{definition}
A finite-dimensional (over $\mathbf k$) graded algebra $A$ is called \textit{$n$-Calabi--Yau} if there exists a linear functional $\mathrm{tr}\colon A^n\to\mathbf k$, where $A^n$ denotes the degree $n$ part of $A$, which is
\begin{enumerate}
    \item Symmetric: $\mathrm{tr}(ab)=(-1)^{|a||b|}\mathrm{tr}(ba)$
    \item Non-degenerate: $(a,b)\mapsto\mathrm{tr}(ab)$ is a non-degenerate pairing on $A$.
\end{enumerate}
Such a functional defines a functional on the cyclic Hochschild complex. In particular, $\per(A)$ is then a proper/right Calabi--Yau dg-category.
\end{definition}

\begin{definition}\label{def:orientable}
    An S-graph $\Sgh$ is called \textit{orientable} if there is an orientation of the edges such for two consecutive halfedges $i$, $i+1$ the parity of $d(i,i+1)$ is even (resp. odd) if $i$ and $i+1$ both point towards or away from the vertex (resp. in different directions). 
    In the case of a 1-valent vertex $v$, $i=i+1$ and $d(i,i+1)=0$ by definition, so the correct condition is instead that such $v$ has even degree.
\end{definition}

\begin{lemma}
Suppose $\Sgh$ arises as the dual S-graph of a mixed-angulation of a weighted marked surface $\sow$ with grading $\nu$. 
Then the orientability of $\Sgh$ in the above sense is equivalent to the orientability of $\nu$ as a foliation.
\end{lemma}
\begin{proof}
Suppose $\nu$ is orientable and an orientation has been chosen, thus all leaves of $\nu$, in particular the edges of the mixed-angulation $\sow$ (i.e.~the polygons), are oriented. To see this, note that every leaf of $\nu$ is homotopic to such an edge. This orientation alternates as one goes around the boundary of any polygon.
Since the edges of the S-graph are transverse to the edges of the mixed-angulation, and the surface is oriented, this induces an orientation of the edges of $\Sgh$. The property that the orientation alternates along the boundary of a polygon is equivalent to the required property of the integers $d(i,i+1)$ (which is the number of vertices of the dual polygon between the halfedges $i$ and $i+1$). Reversing the above steps, we also see that an orientation of the S-graph induces an orientation of $\nu$.
\end{proof}

\begin{proposition}\label{prop:nCY}
    Let $\Sgh$ be an S-graph all of whose vertices are interior. Suppose that either $n$ is odd or that $\Sgh$ is orientable in the above sense.
    Then $A(\Sgh,n)$ is $n$-Calabi--Yau.
\end{proposition}

\begin{proof}
First, suppose $n$ is odd.
Define $\mathrm{tr}\colon A(\Sgh,n)\to\mathbf k$ by $\mathrm{tr}(c_i)\coloneqq 1$ for any halfedge $i$.
Suppose $a$ is a composition of consecutive $a_j$'s and $b$ is the complementary composition of $a_j$'s in the sense that $ab=ba=c_i$.
Then $|a|+|b|=n$, thus $|a||b|=0\mod 2$ by our parity assumption, so $\mathrm{tr}(ab)=1=\mathrm{tr}(ba)=(-1)^{|a||b|}\mathrm{tr}(ba)$.
This shows symmetry.
Non-degeneracy is also clear since for any $a$ which is a concatenation of $a_j$'s we can find $b$ as above and $c_i$ pairs non-trivially with the corresponding idempotent in $A(\Sgh,n)$.

Suppose now that $n$ is even and that a suitable orientation of the edges has been chosen.
For any halfedge $i$ let $\varepsilon(i)=0$ (resp. $1$) if $i$ points away from (resp. towards) its parent vertex. For halfedges $i,j$ at a common vertex, we then find $d(i,j)=\varepsilon(i)+\varepsilon(j)\mod 2$ by our assumption on the orientation.
Define $\mathrm{tr}(c_i)\coloneqq (-1)^{\varepsilon(i)}$.
Then $\mathrm{tr}$ is symmetric, noting that for $a,b$ as in the previous paragraph, $|a|+|b|=n$ is now even.
 \end{proof}

\begin{remark}\label{rem:noCY}
For even $n$, $A(\Sgh,n)$ does not have a Calabi--Yau structure in general, even after possibly modifying the signs in the definition of $A(\Sgh,n)$.
To see this, suppose that $n\neq 0$ is even and $i$ is a halfedge attached to a vertex $v$ of odd degree $m=d(v)$.
Let
\[
s_i\coloneqq a_{i-1}\cdots a_{i+1}a_{i}
\]
be the cycle around the vertex $v$, starting and ending at $i$, and going once around $v$.
Then $c_i=s_i^{n/m}$ by definition and we let $t_i=s_i^{n/m-1}$ so that $c_i=s_it_i=t_is_i$.
Non-degeneracy implies $\mathrm{tr}(c_i)\neq 0$ and symmetry $\mathrm{tr}(s_it_i)=-\mathrm{tr}(t_is_i)$, since both $|s_i|$ and $|t_i|$ are odd by assumption, a contradiction if $\mathrm{char}(\mathbf k)\neq 2$.
\end{remark}

\subsection{Koszul dual}
\label{subsec_koszuldual}

In this subsection we give an explicit description of the Koszul dual dg-algebra of $A(\Sgh,n)$.
First, we recall the construction of the Koszul dual of a general finite-dimensional dg-algebra.
Our main references are~\cite{lefevre_hasegawa,keller_koszul,vdb15}.
Fix a field $\mathbf k$ and let $R\coloneqq \mathbf k^q$, considered as a $\mathbf k$-algebra for some positive integer $q$.
Suppose $A=(A^\bullet,\cdot,d)$ is a dg-algebra with the structure of an augmented $R$-algebra: $A=R\oplus \overline{A}$ where $\overline A\subset A$ is a (non-unital) sub-dg-algebra. To simplify the discussion and since this is sufficient for our purposes, we assume that $\dim_{\mathbf k}A<\infty$ and that any element of $\overline{A}$ is nilpotent.

\begin{definition}
    The \textit{Koszul dual} of an augmented dg $R$-algebra $A=R\oplus\overline{A}$ is the augmented dg $R$-algebra
    \[
    A^!\coloneqq\mathrm{Cobar}(\overline A^{\vee})=\bigoplus_{i\geq 0}\underbrace{\overline A^{\vee}[1]\otimes \cdots \otimes \overline A^{\vee}[1]}_{i\text{ factors}}.
    \]
    where $\overline A^{\vee}$ is the R-dual dg-coalgebra, of which we take the standard (augmented) co-bar resolution.
\end{definition}

To describe the Koszul dual more explicitly, choose a $\mathrm k$-basis $e_1,\ldots,e_n$ of $\overline A$ and let $e^1,\ldots,e^n$ be the dual basis of $\overline A^{\vee}$.
Let $d^i_j\coloneqq e^j(de_i)$ and $m^{i,j}_k\coloneqq e^k(e_i\cdot e_j)$ be the structure constants of $A$.
Then $A^!$ is the free associative $R$-algebra with generators $e^i$ of degree $|e^i|\coloneqq 1-|e_i|$ and differential defined on generators by
\[
de^k\coloneqq -\sum_i d^i_ke^i+\sum_{i,j}(-1)^{|e^i|}m^{j,i}_k e^i\otimes e^j
\]
and extended by the graded Leibniz rule.
Here, the flip $(i,j)\to (j,i)$ comes from passing to the coalgebra and the sign comes from the cobar construction (c.f.~\cite[Subsection 1.2.2]{lefevre_hasegawa} or~\cite[Appendix A]{vdb15}).
Note also that a free associative $R$-algebra is isomorphic to the path algebra of a quiver with $q=\dim_{\mathrm{k}}R$-many vertices.

The part of Koszul duality of interest to us is the following result.
It is essentially contained in \cite{lefevre_hasegawa} but not explicitly stated there.

\begin{proposition}
    Let $A$ be a finite-dimensional dg-algebra, augmented over $R$, such that any element of $\overline{A}$ is nilpotent.
    Consider $R$ as an object in the derived dg-category of $A^!$-modules.
    Then $\mathrm{End}_{A^!}(R)$ is quasi-isomorphic to $A$.
\end{proposition}

\begin{proof}
    We use definitions and notation from~\cite{lefevre_hasegawa}.
    According to~\cite[Theorem 2.2.2.2]{lefevre_hasegawa}, there is an equivalence of derived categories
    \[
    D(A^!)\to D(A^\vee)
    \]
    where $A^\vee$ is considered as a dg-coalgebra, which is co-complete by our assumption, and $D(A^\vee)$ is the coderived category, i.e. the category of dg-comodules localized along certain weak equivalences.
    This functor sends $R$ to $R\otimes_{\tau}A^\vee=A^\vee$, where $\tau:A^\vee\to A^!$ is the universal twisting cochain.
    We note that $A^\vee$ is fibrant and cofibrant as a dg-comodule over itself in the relevant model structure.
    Thus
    \[
    \mathrm{End}_{A^!}(R)\simeq\mathrm{End}_{A^\vee}(A^\vee)=\mathrm{Hom}_R(A^\vee,R)=A.
    \]
\end{proof}

We return to the example of RGB algebras.
The dg-algebra $A(\Sgh,n)$ is $R=\mathbf k^q$-augmented, with $q$ the number of edges of $\Sgh$. $A(\Sgh,n)$ has the following basis.
\begin{enumerate}
    \item $e_i$ for each edge $i$ of $\Sgh$ (constant path), $|e_i|=0$.
    \item For each pair of halfedges $i,j$ attached to an internal vertex $v$ of $\Sgh$, $0\leq r<\frac{n}{m}$ where $m=\deg(v)$, and $i\neq j$ if $r=0$:
    \[
        a^r_{i,j}\coloneqq a_{i-1}\cdots a_{j+1}a_j\left(a_{j-1}\cdots a_{j+1}a_j\right)^r
    \]
    with $|a^r_{i,j}|=d(j,i)+rm$.
    \item For each edge $\{i,j\}$ attached to an internal vertex: $c_i=(-1)^{n-1}c_j$ of degree $n$ if $i$ and $j$ are both attached to internal vertices, or just $c_i$ if $i$ is attached to an internal vertex and $j$ to a boundary vertex. Note that in the first case we make a choice of orientation of the edge, the Koszul duals below of the two different choices are related by a canonical dg-isomorphism.
    \item For each pair of halfedges $i<j$ attached to a boundary vertex:
    \[
        a_{i,j}\coloneqq a_{i-1}\cdots a_{j+1}a_j
    \]
    with $|a_{i,j}|=d(j,i)$.
    \item For each edge $\{i,j\}$ attached to a boundary vertex: $\tau_i=(-1)^{n}\tau_j$ of degree $n-1$ if $i$ and $j$ are both attached to boundary vertices, or just $\tau_i$ if $i$ is attached to an boundary vertex and $j$ to an internal vertex.
    \item $b_{i,j}\coloneqq a_{i,j}\tau_j$ with $|b_{i,j}|=d(j,i)+n-1$ for each pair of halfedges $i<j$ attached to a boundary vertex.
    For notational purposes we also write $b_{i,i}:=\tau_i$.
\end{enumerate}

The Koszul dual dg-algebra $A(\Sgh,n)^!$ has underlying algebra given by the quiver with the same vertices and arrows corresponding to the following dual generators (going in the clockwise direction).

{\renewcommand{\arraystretch}{1.4}\setlength{\tabcolsep}{8pt}
\begin{center}
\begin{tabular}{cl@{\hskip 5em}cl}
     basis element & degree & dual generator & dual degree \\ \hline
     $a^r_{i,j}$ & $d(j,i)+rm$ & $\alpha^r_{i,j}$ & $1-d(j,i)-rm$ \\
     $c_i$ & $n$ & $\sigma_i$ & $1-n$ \\
     $a_{i,j}$ & $d(j,i)$ & $\alpha_{i,j}$ & $1-d(j,i)$ \\
     $\tau_i$ & $n-1$ & $t_i$ & $2-n$ \\
     $b_{i,j}$ & $d(j,i)+n-1$ & $\beta_{i,j}$ & $2-n-d(j,i)$ \\
\end{tabular}
\end{center}
}

The differential has the following non-zero terms:
\[
d\alpha_{i,j}^r=\sum_{\substack{i\leq k\leq j \\ 0\leq s\leq r}}(-1)^{|\alpha_{k,j}^s|}\alpha_{k,j}^s\otimes \alpha_{i,k}^{r-s}
+\sum_{\substack{j\leq k\leq i \\ 0\leq s< r}}(-1)^{|\alpha_{k,j}^s|}\alpha_{k,j}^s\otimes \alpha_{i,k}^{r-s-1}
\]
where we require $i<k$ if $r-s=0$ and $k<j$ if $s=0$ in both sums,
\[
d\sigma_i=\begin{cases}
    A_i+(-1)^{n-1}A_j & i,j \text{ interior} \\
    A_i+(-1)^{n-1}t_j & i \text{ interior}, j\text{ boundary}
\end{cases}
\]
where $\{i,j\}$ is an edge and
\[
A_i\coloneqq \sum_{\substack{0\leq r< \frac{n}{m} \\ j\neq i}}(-1)^{|\alpha_{j,i}^r|}\alpha_{j,i}^{r}\otimes\alpha_{i,j}^{\frac{n}{m}-r-1}+\sum_{0<r<\frac{n}{m}}(-1)^{|\alpha_{i,i}^r|}\alpha_{i,i}^r\otimes \alpha_{i,i}^{\frac{n}{m}-r}
\]
where $i$ is attached to a vertex of degree $m$,
\[
d\alpha_{i,j}=\sum_{j<k<i}(-1)^{|\alpha_{k,j}|}\alpha_{k,j}\otimes\alpha_{i,k},
\]
\[
d\beta_{i,j}= \sum_{j\leq k< i}(-1)^{|\beta_{k,j}|}\beta_{k,j}\otimes\alpha_{i,k}-\sum_{j< k\leq i}\alpha_{k,j}\otimes\beta_{i,k}
\]
for halfedges $i\neq j$ attached to the same boundary vertex, where we set $\beta_{i,i}=t_i$ for ease of notation. The second equation comes from $b_{i,j}=a_{i,k}b_{k,j}=(-1)^{|a_{k,j}|}b_{i,k}a_{k,j}$ and $(-1)^{|\alpha_{k,j}|+|a_{k,j}|}=-1$.

\begin{example}
We return to the example shown on the left in \Cref{fig:2examples} for $n=3$.
The dual dg-algebra $A(\Sgh,3)^!$ is described by the quiver
\[
\begin{tikzcd}[column sep=huge, row sep=huge]
    1 \arrow[loop below,"\alpha_{11}^1"]\arrow[loop left,"\alpha_{11}^2"]\arrow[loop above,"\alpha_{11}^3=\sigma_1"]\arrow[r,swap,bend right=10,"\alpha_{21}"] & 2 \arrow[loop right,"\alpha_{22}^1=\sigma_2"] \arrow[l,swap,bend right=10,"\alpha_{12}"]
\end{tikzcd}
\]
where the arrows have degrees $|\alpha_{11}|=|\alpha_{21}|=0$, $|\alpha_{11}^2|=|\alpha_{12}|=-1$, $|\sigma_1|=|\sigma_2|=-2$.
The generators with non-zero differential are
\[
d\alpha_{11}^2=\alpha_{11}^1\otimes \alpha_{11}^1,\qquad
d\sigma_1=\alpha_{11}^1\otimes\alpha_{11}^2-\alpha_{11}^2\otimes\alpha_{11}^1-\alpha_{12}\otimes\alpha_{21},\qquad d\sigma_2=\alpha_{21}\otimes\alpha_{12}.
\]
\end{example}

\begin{example}
We return to the example shown on the right in \Cref{fig:2examples} for $n=2$.
The dual dg-algebra $A(\Sgh,2)^!$ is described by the quiver
\begin{center}
\begin{tikzcd}[column sep=huge, row sep=huge]
    1 \arrow[loop above,"t_1"] & 2 \arrow[loop right,"\alpha_{22}=\sigma_2"] \arrow[l,swap,bend right=10,"\alpha_{12}"]\arrow[l,bend left=10,"\beta_{12}"] \arrow[loop above,"t_2"]
\end{tikzcd}    
\end{center}
where arrows have degrees $|\alpha_{12}|=|t_1|=|t_2|=0$, $|\sigma_2|=|\beta_{12}|=-1$.
The generators with non-zero differential are
\[
d\sigma_2=-t_2,\qquad d\beta_{12}=t_1\otimes\alpha_{12}-\alpha_{12}\otimes t_2.
\]
\end{example}

\subsection{Examples and coverings}
We spell out some examples of the dg-algebra $A(\Sgh,n)^!$.

\begin{example}\label{ex:relGinzburg}
We consider the disc as a weighted marked surface with three interior singular points of degree $3$ and four boundary singular points and set $n=3$. A mixed-angulation with dual S-graph $\Sgh$ is depicted on the left of \Cref{fig:f.flip.}. Performing the backward flip at the top left edge of $\Sgh$ yields the S-graph $\Sgh^\flat$ depicted on the right in \Cref{fig:f.flip.}.

\begin{figure}[ht]\centering
\makebox[\textwidth][c]{
\begin{tikzpicture}[scale=.5,rotate=0]
\draw[thick](0,0) circle (5);
\foreach \j in {0,...,4}{
    \draw[dashed,white,very thick]
    (54+72*\j:5)arc (54+72*\j:54+72*\j+18:5)
    (54+72*\j:5)arc (54+72*\j:54+72*\j-18:5);}
\foreach \j in {0,...,4}{
    \draw[font=\small] (90+72*\j:5) coordinate  (w\j) edge[blue, thick] (18+72*\j:5)
    (72*\j+5:5)\nn(72*\j-5:5)\nn(72*\j-36-5:5)\nn(72*\j-36+5:5)\nn;}
\draw[blue, thick](w3)to(w0)to(w2);
\foreach \j in {0,...,4}{\draw(w\j)\nn;}
\draw[red,thick,font=\tiny](0,-5) node[below]{$\infty$}\ww to (0,0) node[above]{$3$}\ww to (90-72:3.2) node[right]{$3$}\ww to (90-36:5) node[above]{$\infty$}\ww (90-72:3.2)\ww to (90-108:5) node[right]{$\infty$}\ww;
\draw[red,thick,font=\tiny](0,0) \ww to (90+72:3.2) node[left]{$3$}\ww to (90+36:5) node[above]{$\infty$}\ww (90+72:3.2)\ww to(90+108:5)node[left]{$\infty$}\ww;
\begin{scope}[shift={(14,0)}]
\draw[thick](0,0) circle (5);
\foreach \j in {0,...,4}{
    \draw[dashed,white,very thick]
    (54+72*\j:5)arc (54+72*\j:54+72*\j+18:5)
    (54+72*\j:5)arc (54+72*\j:54+72*\j-18:5);}
\draw[font=\small] (90+72:5) coordinate  (w1) edge[white] (18+72:5);
\foreach \j in {0,2,3,4}{
    \draw[font=\small] (90+72*\j:5) coordinate  (w\j) edge[blue, thick] (18+72*\j:5)
    (72*\j+5:5)\nn(72*\j-5:5)\nn(72*\j-36-5:5)\nn(72*\j-36+5:5)\nn;}
\draw[blue,thick](72+36-5:5).. controls +(-140:4) and +(108:3) ..(72*2+90:5);
\draw(72+36-5:5)\nn;
\draw[blue, thick](w3)to(w0)to(w2);
\draw[red](0,0) to (90-72:3.2);
\foreach \j in {0,...,4}{\draw(w\j)\nn;}
\draw[red,thick,font=\tiny](90-72:3.2)\ww to (90-36:5)\ww (90+72:3.2) node[below]{$3$}\ww to (90+36:5) node[above]{$\infty$}\ww to[bend left=-10] (90+36+72:5) node[left]{$\infty$}\ww (90+72:3.2)\ww to (0,0)\ww;
\draw[red,thick,font=\tiny](0,-5) node[below]{$\infty$}\ww to (0,0) node[above]{$3$}\ww to (90-72:3.2) node[right]{$3$}\ww to (90-36:5) node[above]{$\infty$}\ww (90-72:3.2)\ww to (90-108:5) node[right]{$\infty$}\ww;
\end{scope}
\end{tikzpicture}}
\caption{A mixed angulation with dual S-graph (left), as well as its backward flip at the top left edge (right). The integers denote the degrees of the singular points.}
\label{fig:f.flip.}
\end{figure}

The quiver for the Koszul dual dg RGB algebras $A(\Sgh,3)^!$ and $A(\Sgh^\flat,3)^!$ are shown in 
Figure~\ref{fig:Q1} and Figure~\ref{fig:Q2}, respectively.

\begin{figure}[ht]\centering
\begin{tikzcd}[column sep=huge, row sep=huge]
1 \arrow[rd, "{\alpha_{1,2}}", bend left=15, shift left] \arrow[dd, "{\alpha_{1,3}}", shift left] &                                                                                                                                                                                                                                                &                                                                                                &                                                                                                                                                                                                                                                & 6 \arrow[ld, "{\alpha_{6,4}}", shift left] \arrow[dd, "{\alpha_{6,7}}", bend left=15, shift left] \\
& 2 \arrow[ld, "{\alpha_{2,3}}", bend left=15, shift left] \arrow[lu, "{\alpha_{2,1}}", shift left] \arrow[rr, "{\alpha_{2,4}}", bend left=15, shift left] \arrow[rd, "{\alpha_{2,5}}", shift left] \arrow["\sigma_2"', loop, distance=4em, in=95, out=50] &                                                                                                & 4 \arrow[ll, "{\alpha_{4,2}}"] \arrow[ld, "{\alpha_{4,5}}", bend left=15, shift left] \arrow[ru, "{\alpha_{4,6}}", bend left=15, shift left] \arrow[rd, "{\alpha_{4,7}}", shift left] \arrow["\sigma_4"', loop, distance=4em, in=130, out=85] &                                                                                                \\
3 \arrow[uu, "{\alpha_{3,1}}", bend left=15, shift left] \arrow[ru, "{\alpha_{3,2}}", shift left] &                                                                                                                                                                                                                                                & 5 \arrow[ru, "{\alpha_{5,4}}", shift left] \arrow[lu, "{\alpha_{5,2}}", bend left=15, shift left] &                                                                                                                                                                                                                                                & 7 \arrow[uu, "{\alpha_{7,6}}", shift left] \arrow[lu, "{\alpha_{7,4}}", bend left=15, shift left]
\end{tikzcd}
\caption{The quiver for $A(\Sgh,3)^!$ corresponding to the left S-graph in Figure~\ref{fig:f.flip.}.}
\label{fig:Q1}
\end{figure}

\begin{figure}[ht]\centering
\begin{tikzcd}[column sep=huge, row sep=huge]
 1 \arrow[rd, "{\alpha_{1,2}'}", bend left=15] \arrow[dd, "{\alpha_{1,3}'}"', shift right] \arrow[dd, "{\beta_{1,3}}", shift left]&                                                                                                                                                                                                                                                &                                                                                                &                                                                                                                                                                                                                                                & 6 \arrow[ld, "{\alpha_{6,4}}", shift left] \arrow[dd, "{\alpha_{6,7}}", bend left=15, shift left] \\
& 2 \arrow[lu, "{\alpha_{2,1}'}", description, bend left=15] \arrow[rr, "{\alpha_{2,4}}", bend left=15, shift left] \arrow[rd, "{\alpha_{2,5}}", shift left] \arrow["\sigma_2"', loop, distance=4em, in=95, out=50] &                                           & 4 \arrow[ll, "{\alpha_{4,2}}"] \arrow[ld, "{\alpha_{4,5}}", bend left=15, shift left] \arrow[ru, "{\alpha_{4,6}}", bend left=15, shift left] \arrow[rd, "{\alpha_{4,7}}", shift left] \arrow["\sigma_4"', loop, distance=4em, in=130, out=85] &                                                                                                \\
3 \arrow["{\beta_{3,3}}"', loop, distance=2em, in=215, out=145]     &                                                                                                                                                                                                                                                & 5 \arrow[ru, "{\alpha_{5,4}}", shift left] \arrow[lu, "{\alpha_{5,2}}", bend left=15, shift left] &                                                                                                                                                                                                                                                & 7 \arrow[uu, "{\alpha_{7,6}}", shift left] \arrow[lu, "{\alpha_{7,4}}", bend left=15, shift left]
\end{tikzcd}
\caption{The quiver for $A(\Sgh,3)^!$ corresponding to the right S-graph in Figure~\ref{fig:f.flip.}.}
\label{fig:Q2}
\end{figure}
\end{example}

\begin{example}\label{ex:quot1}
Consider the disc $\sow$ with six interior singular points of degrees $1,2,4,4,4,4$ and let $n=4$. The weighted marked surface $\sow$ admits a mixed-angulation with dual S-graph $\Sgh$, see the right picture of \Cref{fig:CY4}.

\begin{figure}[ht]
\begin{tikzpicture}[scale=1.5]
\begin{scope}[]
\clip(-2.2,2.2)rectangle(2.2,-2.2);
\draw (1,1) rectangle (-1,-1);
\foreach \j in {1,2,3,4}{
\begin{scope}[shift={(90*\j:2)},rotate=90*\j-90]
  \draw[blue,thick]
    (-1,1)\nn to(1,1)\nn
        to[bend left=-50](1,-1)\nn to(-1,-1)\nn
        to[bend left=-50](-1,1)\nn;
  \draw[blue,thick](-45:1.414)to[bend left=45]
        (-45+30:1.414)to[bend left=45]
        (-45+60:1.414)to[bend left=45](-45+90:1.414);
  \draw[thick](-45:1.414)arc(-45:45:1.414);
  \foreach \k in {0,...,9}
    {\draw[](-45+10*\k:1.414)\nn;}
\end{scope}
\begin{scope}[shift={(90*\j:2)},rotate=90*\j+90]
  \foreach \k in {0,...,9}{\draw(-45+10*\k:1.414)\nn;}
  \draw[blue,thick](-45:1.414)to[bend left=45]
    (-45+30:1.414)to[bend left=45]
    (-45+60:1.414)to[bend left=45](-45+90:1.414);
  \draw[thick](-45:1.414)arc(-45:45:1.414);
\end{scope}}
\draw[white,fill=white](2.2,2.2)rectangle(-2.2,2)(-2.2,-2.2)rectangle(2.2,-2)
    (-2.2,2)rectangle(-2,-2.2)(2.2,2.2)rectangle(2,-2.2);
\draw[opacity=.5,orange,line width=1.2mm,->-=.6,>=stealth](-1.414,-2)to(1.414,-2);
\draw[opacity=.5,orange,line width=1.2mm,->-=.6,>=stealth](-1.414,2)to(1.414,2);
\draw[opacity=.5,Emerald,line width=1.2mm,->-=.6,>=stealth](-2,-1.414)to(-2,1.414);
\draw[opacity=.5,Emerald,line width=1.2mm,->-=.6,>=stealth](2,-1.414)to(2,1.414);

\foreach \j in {1,2,3,4}{
\begin{scope}[shift={(0,0)},rotate=90*\j]
  \draw[red,thin]
    (2,1)\ww to[bend left=15](1.35,1.15)\ww (2,-1)\ww to[bend left=15](1.35,-1.15)\ww
    (2,1.3)\ww to(2,-1.3)\ww (2,1)\ww to(2,-1)\ww  (0,0)\ww to(2,0)\ww;
\end{scope}}
\draw[gray,thick,-stealth](-45:.5)arc(-45:135:.5);
\draw[](45:.6)node[rotate=-45,font=\tiny]{$\pi/2$};
\end{scope}
\draw(2.8,0)node{$\xrightarrow[\text{4-covering}]{\text{branched}}$};
\begin{scope}[shift={(5.5,0)},scale=.4]
\foreach \j in {0,...,8}{\draw (40*\j:5) coordinate  (w\j) edge[very thick] (40+40*\j:5);}
\draw[blue,thick](w0)to(w0)to[bend left=-15](w3)to[bend left=-15](w6)to[bend left=-15](w0)
    (5,0).. controls +(160:3) and +(90:.5) ..(1.25,0).. controls +(-90:.5) and +(200:3) ..(5,0)
    (5,0).. controls +(140:5) and +(90:1) ..(-1.25,0).. controls +(-90:1) and +(220:5) ..(5,0);
\foreach \j in {0,...,8}{\draw(w\j)\nn;}
\draw[red,font=\tiny](-4,0)\ww node[above]{$4$}to(0,0)\ww node[above]{$2$}to(2.5,0)\ww node[right]{$1$}
    (45:4)\ww node[above]{$4$}to[bend left=-30](-2.5,0)\ww node[above]{$4$}to[bend left=-30](-45:4)\ww node[below]{$4$};
\end{scope}
\end{tikzpicture}
\caption{A mixed-angulation with dual S-graph of the disc $\sow$ (on the right) arising from folding a 4-angulation $\widetilde{\mathbb A}$ of a tours $\widetilde{\sow}$ with one boundary component (on the left, where the opposite orange/green edges are glued together)}\label{fig:CY4}
\end{figure}

The quiver for the dual RGB algebra $A(\Sgh,4)^!$ is shown in \Cref{fig:Q3}.

\begin{figure}[!ht]\centering
\begin{tikzcd}[column sep=huge, row sep=huge]
& 3 \arrow[ld, "{\alpha_{3,2}^0}" description] \arrow[rd, "{\alpha_{3,4}^0}", bend left] \arrow[dd, "{\alpha_{3,1}^0}", bend left=15] \arrow["\sigma_3"', loop, distance=2em, in=125, out=55]  &                                                                                                                                                                                                                                                                &  &                                                                                                                                                                                                                                                                                                                                              \\
2 \arrow[ru, "{\alpha_{2,3}^0}", bend left] \arrow[rd, "{\alpha_{2,1}^0}" description] \arrow[rr, "{\alpha_{2,4}^0}", bend left=15] \arrow["\sigma_2"', loop, distance=2em, in=215, out=145] &                                                                                                                                                                          & 4 \arrow["{\alpha_{4,4}^1}"', loop, distance=4em, in=305, out=250] \arrow[lu, "{\alpha_{4,3}^0}" description] \arrow[ld, "{\alpha_{4,1}^0}", bend left] \arrow[ll, "{\alpha_{4,2}^0}", bend left=15] \arrow[rr, "{\alpha_{4,5}^0}", bend left=15] \arrow[rr, "{\alpha_{4,5}^1}", bend left=45] \arrow["\sigma_4"', loop, distance=4em, in=95, out=50] &  & 5 \arrow[ll, "{\alpha_{5,4}^0}", bend left=15] \arrow[ll, "{\alpha_{5,4}^1}", bend left=45] \arrow["{\alpha_{5,5}^1}"', loop, distance=3em, in=295, out=235] \arrow["{\alpha_{5,5}^2}"', loop, distance=3em, in=0, out=305] \arrow["{\alpha_{5,5}^3}"', loop, distance=3em, in=65, out=10] \arrow["\sigma_5"', loop, distance=3em, in=125, out=70] \\
& 1 \arrow[ru, "{\alpha_{1,4}^0}" description] \arrow[lu, "{\alpha_{1,2}^0}", bend left] \arrow[uu, "{\alpha_{1,3}^0}", bend left=15] \arrow["\sigma_1"', loop, distance=2em, in=305, out=235] &                                                                                                                                                                                                                                                                &  &
\end{tikzcd}
\caption{The quiver for $A(\Sgh,4)^!$ corresponding to the right picture of Figure~\ref{fig:CY4}.}
\label{fig:Q3}
\end{figure}
\end{example}

\begin{remark}
In \Cref{ex:quot1}, there is a branched 4-covering of the mix-angulation $\mathbb A$ of $\sow$
as shown in \Cref{fig:CY4}.
This is also true in general, cf. the proposition below.
\end{remark}

\def\gp{K}
\begin{proposition}\label{pp:covering}
For any mixed-angulation $\mathbb{A}$ (with dual S-graph $\Sgh$) of $\sow$,
there is a branching covering $p\colon\widetilde{\sow}\to\sow$ such that $\mathbb{A}$ lifts to an $n$-angulation\footnote{By which we mean a mixed-angulation consisting of $n$-gons and $\infty$-gons.} $\widetilde{\mathbb A}$ (with dual S-graph $\widetilde{\Sgh}$), where the quotient group is a finite group. On the level of associated dg-algebras,
$A(\Sgh,n)$ and the dual $A(\Sgh,n)^!$ are quotients of 
$A(\widetilde{\Sgh},n)$ and the dual $A(\widetilde{\Sgh},n)^!$, respectively.
\end{proposition}
\begin{proof}
Let $\Delta^\circ$ be the subset of $\Delta$, consisting of interior singular points, i.e. 
points with finite weights/degrees. Consider the fundamental group $\pi_1(\sow\setminus\Delta^\circ)$ for any chosen base point $X$ in the interior of $\sow\setminus(\M\cup\Delta^\circ)$.
This group contains a free subgroup generated by the loops $l_x$ around the interior singular points $x\in\Delta^\circ$. Take the subgroup of $\pi_1(\sow, X)$, generated by $l_x^{\frac{n}{d(x)}}$, which determines a regular covering of $\sow\setminus\Delta^\circ$ with covering group $\gp=\prod_{x\in\Delta^\circ}\ZZ/{\frac{n}{d(x)}}$.
Such a $\gp$-covering extends to a branched $\gp$-covering 
\[
    p\colon\widetilde{\sow}\to\sow,
\]
branching at points in $\Delta^\circ$. 
The mixed-angulation $\mathbb A$ lifts to a mixed-angulation $\widetilde{\mathbb A}$.
By examining the degree of lifts of points in $\Delta^\circ$ in $\widetilde{\mathbb A}$, 
one sees that they all equal $n$. This gives the desired $n$-angulation $\widetilde{\mathbb A}$. The stated relation between the quivers and dg-algebras is now straightforward to see.
\end{proof}

\section{RGB algebras from $A_\infty$-deformations}
\label{sec_deffuk}

This section is organized as follows.
In \Cref{subsec_construction} we define a deformation of the Fukaya category of a surface over $\mathbf k[[t]]$ in the formalism of curved $A_\infty$-categories.
In \Cref{subsec_transfer} we determine a component of the space of stability conditions of the newly constructed categories in terms of moduli spaces of quadratic differentials. This is based on transfer of stability conditions along an adjunction with the Fukaya category of the surface.
Finally, in \Cref{subsec_sgraphs} we restrict to the case of quadratic differentials with infinite flat area and establish the relation with RGB algebras and thus the schober construction from \Cref{subsec:BGAschober}.

\subsection{Construction of the category}
\label{subsec_construction}

As a starting point, we review the construction of the (partially wrapped) Fukaya category of a surface following~\cite{HKK17}.
The input data is a \textit{graded marked surface} (in the sense below) and a choice of ground field $\mathbf k$.

\begin{definition}\label{def:gradedmarkedsurface}
    A \emph{graded marked surface} is a triple $(S,M,\nu)$ where
    \begin{enumerate}
    \item $S$ is a compact, oriented surface, possibly with boundary $\partial S$.
    \item $M\subset S$ is a finite subset of marked points. These can be both on the boundary and/or the interior of $S$. We require $M$ to be non-empty and every component of $\partial S$ to contain at least one marked point.
    \item $\nu$ is a grading structure on $S\setminus M$, i.e.~a section of the projectivized tangent bundle $\mathbb P(TS)$ of $S$ over $S\setminus M$.
\end{enumerate}
\end{definition}

The first step in the construction involves cutting $S$ along \textit{arcs} into polygons whose vertices belong to $M$.
This is an auxiliary choice, in the sense that the Fukaya category is independent of it, up to quasi-equivalence.

\begin{definition}
    An \emph{arc system} on a graded marked surface $(S,M,\nu)$ is given by a finite collection, $\mathbb X$, of immersed compact intervals $X\colon [0,1]\to S$ so that
    \begin{enumerate}
        \item Endpoints of arcs belong to $M$.
        \item Arcs can intersect themselves and each other only in the endpoints.
        \item Each $p\in M$ is the endpoint of at least one arc and all arcs starting at $p$ should point in different directions in $T_pS$.
        \item The collections of all arcs in $\mathbb X$ cuts $S$ into polygons.
    \end{enumerate}
    Furthermore, we assume that a grading has been chosen for each arc.
\end{definition}

In order to define the $A_\infty$-structure, it is convenient to consider the \textit{real blow-up} of $S$ in $M$, which is a surface with corners $\widehat{S}$, together with a map $\pi\colon \widehat S\to S$.
The surface $\widehat S$ is constructed by replacing each $p\in M\cap \partial S$ by an interval connecting a pair of corners, and replacing each $p\in M\cap \mathrm{int}(S)$ by a new boundary circle.
We refer to $\pi^{-1}(M)\subset \partial\widehat S$, i.e.~the ``new'' part of the boundary, as the \textit{marked boundary}.
We can extend $\pi^{*}\nu$ to $\widehat S$ after possibly perturbing $\nu$ near $M$.
The arcs in a given arc system $\mathbb X$ then lift to disjoint embedded intervals in $\widehat S$.

\begin{definition}
    A \emph{boundary path} in $\widehat S$ is an immersed path $a\colon [0,1]\to \pi^{-1}(M)\subset \partial\widehat S$, up to reparametrization, which follows the boundary in the direction opposite to its induced orientation, i.e.~so that the surface lies to the right.
\end{definition}

A boundary path $a$ which starts at an arc $X\in\mathbb X$ and ends at an arc $Y\in\mathbb X$ has an integer degree $|a|\coloneqq i(X,a)-i(Y,a)\in\ZZ$, where an arbitrary grading on $a$ has been chosen.
This is independent on the choice of grading on $a$ and additive under concatenation.

While boundary paths will be used in the definition of morphisms, structure constants come from counting immersed $2d$-gons in the following sense.

\begin{definition}
    For $d\in\mathbb Z_{>0}$ an \emph{immersed $2d$-gon} is an immersion $\psi\colon D\to\widehat S$ from a $2d$-gon $D$ so that the edges of $D$ are mapped, alternatingly, to arcs in $\mathbb X$ and boundary paths between these arcs.
    The resulting cyclic sequence of boundary paths, $a_d$, $a_{d-1}$, \ldots, $a_1$, as one goes counter-clockwise around the boundary of the polygon, is called the \emph{associated cyclic sequence}.
\end{definition}

The associated cyclic sequence $a_d,a_{d-1},\ldots,a_1$ in fact determines the immersed $2d$-gon, up to reparametrization. Moreover, $|a_d|+\ldots+|a_1|=d-2$, as shown in \cite{HKK17}.
Next, define an $A_\infty$-category whose objects are the chosen arcs. This will be a full subcategory of the Fukaya category $\mathcal F(S,M,\nu)$ which generates it.

\begin{definition}
    Let $(S,M,\nu)$ be a graded marked surface with arc system $\mathbb X$.
    Define a strictly unital $A_\infty$-category $\mathcal F_{\mathbb X}=\mathcal F_{\mathbb X}(S,M,\nu)$ over $\mathbf k$ with
    \begin{itemize}
        \item $\mathrm{Ob}(\mathcal F_{\mathbb X})=\mathbb X$
        \item $\mathrm{Hom}(X,Y)$ is the vector space with basis consisting of boundary paths starting at $X$ and ending at $Y$, and in addition, if $X=Y$, the identity morphism. The $\mathbb Z$-grading on $\mathrm{Hom}(X,Y)$ comes from the degree of boundary paths.
        \item There are three types of terms which contribute to the structure maps, $\mathfrak m_d$:
        \begin{enumerate}
            \item If $a_d,\ldots,a_1$ is a cyclic sequence of boundary paths associated with an immersed $2d$-gon, then this contributes a term $b$ to $\mathfrak m_d(ba_d,a_{d-1},\ldots,a_1)$ and a term $(-1)^{|b|}b$ to $\mathfrak m_d(a_d,\ldots,a_2,a_1b)$.
            \item If $a,b$ are boundary paths such that $a$ starts where $b$ ends, then $\mathfrak m_2(a,b)$ has a term $(-1)^{|b|}ab$ where $ab$ is the concatenated path.
            \item $\mathfrak m_2(a,1)=a=(-1)^{|a|}\mathfrak m_2(1,a)$
        \end{enumerate}
    \end{itemize}
\end{definition}

It is shown in~\cite{HKK17} that $\mathcal F_{\mathbb X}$ is an $A_\infty$-category over $\mathbf k$.
Moreover, the category $\mathrm{Tw}^+(\mathcal F_{\mathbb X})$ of one-sided twisted complexes over $\mathcal F_{\mathbb X}$ --- which represents the closure under shifts, direct sums, and cones --- is up to quasi-equivalence independent of the choice of $\mathbb X$.
More precisely, the classifying space of arc systems is contractible, and $\mathrm{Tw}^+(\mathcal F_{\mathbb X})$ are the fibers of a local system of (pre-)triangulated $A_\infty$-categories on this space.

\begin{definition}
The \emph{Fukaya category of a graded marked surface} $(S,M,\nu)$ is defined as
\begin{equation}
\mathcal F(S)=\mathcal F(S,M,\nu)\coloneqq \mathrm{Tw}^+(\mathcal F_{\mathbb X})\,,
\end{equation}
for any choice of arc system $\mathbb X$.
\end{definition}

The next step, following~\cite{Hai21}, is to enhance the $\mathbb Z$-grading on morphisms in $\mathcal F(S)$ to a $\ZZ\times \ZZ/2$-grading such that structure maps are even with respect to the additional grading.
To this end, let $\Sigma\to S\setminus M$ be the double cover on which $\nu$ becomes an oriented foliation, i.e.~the fiber over a given $p\in S\setminus M$ is the set of orientations of the line $\nu(p)\subset T_pS$.
When choosing an arc system $\mathbb X$, we now also assume that a lift of the interior of each arc to $\Sigma$ has been chosen.
Then one can define the \emph{parity}, $\pi(a)\in\ZZ/2$, of a boundary path $a\colon [0,1]\to\widehat S$ from an arc $X$ to an arc $Y$ to be even (resp. odd) if the endpoints of the lifts of the two arcs are connected by a lift of $a$ to $\Sigma$ (resp. not connected by such a lift).
This defines an additional $\ZZ/2$-grading on $\Hom(X,Y)$ which is compatible with the structure maps.

For later purposes, we note that the choice of grading and lift to $\Sigma$ of an arc $X$ determine an orientation of $X$ as follows: On the one hand the foliation $\nu$ is oriented along $X$ by the choice of lift to $\Sigma$, on the other hand, the grading provides a homotopy class of paths in $\mathbb P(T_pS)$ between $\nu(p)$ and the tangent space $T_pX$ to $X$ for a given point $p$ on $X$, which thus receives an orientation.

If $a$ is a boundary path from an arc $X_-$ to an arc $X_+$, both given their induced orientations as above, then we define:
\begin{equation}
    \varepsilon_\pm(a)\coloneqq \begin{cases}
        0 & X_\pm \text{ points away from } a \\
        1 & X_\pm \text{ points towards } a
    \end{cases}
\end{equation}
Then $|a|+\pi(a)=\varepsilon_-(a)+\varepsilon_+(a)\mod 2$.

There are a-priori two different categories of twisted complexes attached to the $\ZZ\times\ZZ/2$-graded version of $\mathcal F_{\mathbb X}$.
The first, denoted $\mathrm{Tw}^+_{\ZZ/2}(\mathcal F_{\mathbb X})$, is again $\ZZ\times\ZZ/2$-graded, but closed under cones over \textit{homogeneous} (either even or odd) closed morphisms only.
The second, denoted $\mathrm{Tw}^+(\mathcal F_{\mathbb X})$, is the usual construction of one-sided twisted complexes, ignoring the additional grading, and thus has cones of closed morphisms with both an even and odd component.
Equivalently, viewing the additional $\ZZ/2$-grading as an involutive functor, $I$, on $\mathcal F_{\mathbb X}$ which fixes objects and acts by $(-1)^{\pi(a)}$ on morphisms, $\mathrm{Tw}^+_{\ZZ/2}(\mathcal F_{\mathbb X})$ is the subcategory of $\mathrm{Tw}^+(\mathcal F_{\mathbb X})$ of objects which are fixed under the functor which is the natural extension of $I$ to the category of twisted complexes.
Nevertheless, it turns out that the two categories are quasi-equivalent.

\begin{lemma}
The inclusion of $\mathrm{Tw}^+_{\ZZ/2}(\mathcal F_{\mathbb X})$ into $\mathrm{Tw}^+(\mathcal F_{\mathbb X})$ is a quasi-equivalence of $A_\infty$-categories.
Thus, the involutive functor defined by the additional $\ZZ/2$-grading fixes all isomorphism classes of objects in $\mathcal F(S,M,\nu)$.
\end{lemma}

\begin{proof}
This is shown in \cite{Hai21} in the case when $\partial S=\emptyset$, however that assumption is unnecessary and the same proof works in the case with boundary.
\end{proof}

The definition of the Fukaya category $\mathcal F(S,M,\nu)$ is based on counting immersed polygons in $\widehat S$.
Instead, we could also count immersed polygons in $S$, or equivalently immersed polygons ``with holes'' in $\widehat S$, where each hole corresponds to a point in the polygon in $S$ which maps to a point in $M$.
Algebraically, this turns out to give a deformation of $\mathcal F(S,M,\nu)$ over $\mathbf k[[t]]$, where the degree of $t$ has to be a positive multiple of an integer which depends on the behaviour of $\nu$ near $M$.
In the following we discuss all this in more detail.

\begin{definition}
Let $(S,M,\nu)$ be a graded marked surface, $M'\subset M\cap\mathrm{int}(S)$, and $n\in\ZZ$.
The pair $(M',n)$ is \emph{compatible} if for every $p\in M'$ there is a positive integer $m$ with $n=\mathrm{ind}_\nu(p)m$.
(The index $\mathrm{ind}_\nu(p)$ is the winding number of $\nu$ around $p$.)
\end{definition}

Note that in the situation of the definition above, $n/\mathrm{ind}(p)$ is a well-defined positive integer unless $\mathrm{ind}(p)=0$ in which case $n=0$ also. In that case we define $n/\mathrm{ind}(p)\coloneqq 1$.

Fix a graded marked surface $(S,M,\nu)$ and a compatible pair $(M',n)$ in the above sense.
The next step is to construct a curved, $\ZZ\times\ZZ/2$-graded $A_\infty$-category over $\mathbf k[[t]]$, $|t|=2-n$, $\pi(t)=n\mod 2$, whose base change to $\mathbf k$ is $\mathcal F_{\mathbb X}$.
Here, \textit{curved} means that besides the $A_\infty$-operations $\mathfrak m_d$, $d\geq 1$, there are also elements $\mathfrak m_0\in\mathrm{Hom}^2(X,X)$ for every object $X$, satisfying a generalization of the usual $A_\infty$-category equations. Morphism spaces are required to be topologically free $\mathbf k[[t]]$-modules.
We emphasize also that $\mathbf k[[t]]$ is complete and commutative in the $\ZZ\times\ZZ/2$-graded sense, and thus its underlying ungraded algebra is the algebra of polynomials if $|t|\neq 0$, or the algebra of formal power series if $|t|=0$.
We refer the reader to~\cite{Hai21} for the details on the definition of such $A_\infty$-categories.

The following definition generalizes the kinds of immersed polygons in the definition of $\mathcal F_{\mathbb X}$ --- the special case $k=0$.

\begin{definition}
Let $D$ be a $2d$-gon with $k\geq 0$ open disks removed from its interior. Denote by $e_1,f_1,e_2,f_2,\ldots,e_d,f_d$ the edges of $D$ in clockwise order, and by $c_1,\ldots,c_k$ the boundaries of the removed disks.
An \emph{immersed $2d$-gon with $k$ holes} is an immersion $\psi\colon D\to \widehat S$ so that each $e_i$ is mapped to an arc in $\mathbb X$, each $f_i$ is mapped to a boundary path, $a_i$, in $\widehat S$, and each $c_i$ is mapped to a component $\pi^{-1}(p)$ of $\partial\widehat S$ for some $p\in M'$ in a $d/\mathrm{ind}(p)$-to-1 covering.
The \emph{associated cyclic sequence} of boundary paths is $a_d,a_{d-1},\ldots,a_1$.
\end{definition}

\begin{lemma}
If $a_d,a_{d-1},\ldots,a_1$ is the associate cyclic sequence of some immersed $2d$-gon with $k$ holes $\psi\colon D\to\widehat S$, then
\begin{equation}
    |a_1|+\ldots+|a_d|=d-2+(2-n)k, \qquad \pi(a_1)+\ldots+\pi(a_d)=nk\mod 2.
\end{equation}
\end{lemma}

\begin{proof}
This is a slight generalization of \cite[Lemma 3.14]{Hai21}.
The key point is that the foliation $\psi^*\nu$ on $D$ satisfies $\mathrm{ind}(c)=n$ if $c$ is the boundary of one of the $k$ removed disks.
In the case $n=2$ we get no correction from the holes, since $\psi^*\nu$ extends to a smooth foliation on the entire $2d$-gon (i.e.~without the holes) in that case.
In general, the Poicar\'e--Hopf theorem yields $\mathrm{ind}(c)=2+(n-2)k$ for a simple loop $c$ which goes around all the $k$ holes in $D$ in counterclockwise direction.
\end{proof}

\begin{lemma}
Fix a cyclic sequence $a_d,a_{d-1},\ldots,a_1$ of boundary paths and $k\geq 0$, then there are finitely many $2d$-gons with $k$ holes, up to reparametrization, with associated cyclic sequence the given one.
\end{lemma}

Note that in the case $n\neq 2$, $k$ is already determined by the $a_i$'s by the previous lemma, so does not need to be fixed.

\begin{proof}
The proof is essentially the same as in \cite[Lemma 3.15]{Hai21}, with the only difference that the factor ``3'' appearing there needs to be replaced by $n$ if $n\neq 0$ and 1 if $n=0$.
The basic idea is as follows: Let $D$ be some $2d$-gon with $k$ holes and associated cyclic sequence the given one. Our assumptions fix the number of edges of $\mathbb X$-polygons, counted with multiplicity, which appear on the boundary of $D$. But this implies a bound on the number of $\mathbb X$-polygons tessellating $D$.
\end{proof}

\begin{definition}
\label{def_curvedainfty}
Let $(S,M,\nu)$ be a graded marked surface, $\mathbb X$ an arc system, $M'\subset M$, and $n\in\ZZ$ such that $(M',n)$ is compatible.
Define a curved, $\ZZ\times\ZZ/2$-graded $A_\infty$-category $\mathcal A_{\mathbb X}$ over $\mathbf k[[t]]$, $|t|=2-n$, $\pi(t)=n\mod 2$, as follows.
\begin{itemize}
    \item $\mathrm{Ob}(\mathcal A_{\mathbb X})\coloneqq \mathrm{Ob}(\mathcal F_{\mathbb X})=\mathbb X$
    \item $\mathrm{Hom}_{\mathcal A_{\mathbb X}}(X,Y)\coloneqq \mathrm{Hom}_{\mathcal F_{\mathbb X}}(X,Y)[[t]]$
    \item Structure maps: Modulo $t$, these are just the structure maps for $\mathcal F_{\mathbb X}$. There are the following additional terms in higher powers of $t$:
    \begin{enumerate}
        \item For each $X\in\mathbb X$ and endpoint $p$ of $X$ with $p\in M'$ there  is a term $\pm c_{p,X}t$ of $\mathfrak m_0\in\mathrm{Hom}^2(X,X)$, where $c_{p,X}$ is the closed boundary path which starts and ends at $X$ and winds $n/\mathrm{ind}(p)$ times around the component of $\partial \widehat S$ corresponding to $p\in M'$.
        The sign is $+1$ (resp. $(-1)^{n}$) if $X$ points towards (resp. away from) $p$ with respect to the natural orientation coming from the grading of $X$ and lift to $\Sigma$.
        \item Each immersed $2d$-gon with $k$ holes and associated cyclic sequence $a_d$, $a_{d-1}$, \ldots, $a_1$ contributes: a) a term $(-1)^{nk\varepsilon_-(a_1)}1_Xt^k$  to $\mathfrak m_d(a_d,\ldots,a_1)$, where $X$ is the arc where $a_1$ starts, b) a term $(-1)^{nk\varepsilon_-(a_1)}bt^k$ to $\mathfrak m_d(ba_d,\ldots,a_1)$, where $b$ is a boundary path which stars where $a_d$ ends, and c) a term $(-1)^{|b|+nk\varepsilon_-(b)}bt^k$ to $\mathfrak m_d(a_d,\ldots,a_1b)$, where $b$ is a boundary path which ends where $a_1$ starts.
    \end{enumerate}
\end{itemize}
\end{definition}

\begin{proposition}
$\mathcal A_{\mathbb X}=\mathcal A_{\mathbb X}(S,M,\nu,M',n)$ is a curved $A_\infty$-category over $\mathbf k[[t]]$.
\end{proposition}

\begin{proof}
The proof is the same as in~\cite{Hai21} except for the signs.
More precisely, the signs are the same if $n$ is odd and simplify in the case where $n$ is even, so we will omit the (straightforward) checking of signs.
\end{proof}

We briefly recall some algebraic constructions from~\cite{Hai21}.
Given any curved $A_\infty$-category $\mathcal A$ over $\mathbf k[[t]]$ we obtain an uncurved one ($\mathfrak m_0=0$) by passing to the $A_\infty$-category of (not necessarily one-sided) twisted complexes $\mathrm{Tw}(\mathcal A)$.
An object of $\mathrm{Tw}(\mathcal A)$ is a pair $(X,\delta)$, where $X$ is an object in the closure of $\mathcal A$ under finite direct sums and shift, and $\delta\in\mathrm{Hom}^1(X,X)$ has $t^0$ term which is strictly upper-triangular and $\delta$ satisfies the $A_\infty$ Maurer--Cartan equation.
Base change from $\mathbf k[[t]]$ to $\mathbf k$ gives a functor
\begin{equation*}
    F\colon \mathrm{Tw}(\mathcal A)\to\mathrm{Tw}^+\left(\mathcal A\otimes_{\mathbf k[[t]]}\mathbf k\right)
\end{equation*}
which has a right adjoint, $G$, see \cite[Prop.~2.13]{Hai21}.
Moreover, the counit of the adjunction, $FG\to 1$, fits into an exact triangle of functors and natural transformations
\begin{equation}
\label{spherical_is_shift}
    [n-1,n]\longrightarrow FG\longrightarrow 1\longrightarrow [n,n]
\end{equation}
where $[n,k]$ is the shift in bi-degree $(n,k)\in\mathbb Z\times\mathbb Z/2$.
We denote by $\mathrm{Tors}(\mathcal A)$ the full triangulated subcategory of $\mathrm{Tw}(\mathcal A)$ which is generated by the image of $G$.
By restricting $F$ to $\mathrm{Tors}(\mathcal A)$, we obtain an adjunction between $\mathrm{Tw}^+\left(\mathcal A\otimes_{\mathbf k[[t]]}\mathbf k\right)$ and $\mathrm{Tors}(\mathcal A)$.

\begin{remark}
Using more geometric language, $\mathcal A$ is a family of $A_\infty$-categories over a formal disk (with coordinate in degree $2-n$).
Moreover, $\mathcal A_0\coloneqq \mathcal A\otimes_{\mathbf k[[t]]}\mathbf k$ is the central fiber of this family, $\mathrm{Tw}^+(\mathcal A_0)$ are objects living on the central fiber, and $\mathrm{Tors}(\mathcal A)$ are objects supported on the central fiber.
\end{remark}

We apply these algebraic constructions to the curved $A_\infty$-category $\mathcal A_{\mathbb X}$.
It turns out that $\mathrm{Tors}(\mathcal A_{\mathbb X})$ does not depend, up to coherent quasi-equivalence, on the choice of arc system $\mathbb X$.
This can be proven in the same way as in~\cite{HKK17,Hai21}.
Thus, we omit $\mathbb X$ from the notation in the following definition.

\begin{definition}
    $\mathcal C(S,n)=\mathcal C(S,M,\nu,M',n)\coloneqq \mathrm{Tors}(\mathcal A_{\mathbb X}(S,M,\nu,M',n))$ is the triangulated $A_\infty$-category associated with the quintuple $(S,M,\nu,M',n)$.
\end{definition}

By definition, $\mathcal A_{\mathbb X}\otimes_{\mathbf k[[t]]}\mathbf k=\mathcal F_\mathbb X$, so the right adjoint of the base change functor is a functor
\begin{equation}\label{G_functor}
    G\colon \mathcal F(S,M,\nu)=\mathrm{Tw}^+(\mathcal F_{\mathbb X})\to\mathrm{Tors}(\mathcal A_{\mathbb X})=\mathcal C(S,n).
\end{equation}

The following proposition gives an explicit description of $\mathrm{Ext}^i$'s between those objects in $\mathcal C(S,n)$ which correspond to arcs.

\begin{proposition}
\label{prop_extcomplex}
    Let $\mathbb X$ be an arc system for $(S,M)$ which does not cut out any 1-gons, and $X,Y\in\mathbb X$.
    Then $\mathrm{Ext}^\bullet_{\mathrm{Tw}(\mathcal A_{\mathbb X})}(X,Y)$ is the cohomology of the complex
    \begin{equation*}
        \mathrm{Hom}_{\mathcal F_{\mathbb X}}(X,Y)\otimes \on{H}^\bullet(S^{n-1};\mathbf k),\qquad d(a+b\tau)=(-1)^n\mathfrak m_2(b,\mathfrak m_0t^{-1})
    \end{equation*}
    where $\tau\in \on{H}^{n-1}(S^{n-1};\mathbf k)$ is the fundamental class.
\end{proposition}

\begin{proof}
By the formula for the cone over the counit~\eqref{spherical_is_shift}, and adjunction, there is an exact triangle of complexes
\begin{equation}
    \mathrm{Hom}_{\mathcal F_{\mathbb X}}(X,Y)\to \mathrm{Hom}_{\mathrm{Tw}(\mathcal A_{\mathbb X})}(GX,GY)\to \mathrm{Hom}_{\mathcal F_{\mathbb X}}(X[n-1],Y).
\end{equation}
The differential on $\mathrm{Hom}_{\mathcal F_{\mathbb X}}(X,Y)$ is trivial by the assumption on $\mathbb X$.
The formula for the differential can be read off from
\begin{equation*}
    F(G(X))=\left(X\oplus X[n-1,n],\begin{pmatrix} 0 & 0 \\ (-1)^n\mathfrak m_0t^{-1} & 0 \end{pmatrix}\right)
\end{equation*}
which can be found in \cite[Subsection 2.3]{Hai21}.
\end{proof}

\begin{remark}
If $\partial S=\emptyset$, $M'=M$, and either $n$ is odd or $\nu$ orientable, then one can show that $\mathcal C(S,n)$ is proper and $n$-Calabi--Yau in the sense of proper $A_\infty$-categories over $\mathbf k$. The proof is similar to the one in the special case $n=3$ found in~\cite{Hai21}.
\end{remark}

\subsection{Relation to RGB algebras}
\label{subsec_sgraphs}

In this subsection we restrict to the case of quadratic differentials of infinite area and show that a full subcategory of $\mathcal C_{\mathrm{len}}(S,n)$ (often equal to the whole category) has a generator whose Yoneda algebra is an RGB algebra in the sense of Section~\ref{sec_gbga}.

In~\cite[Subsection 2.3]{CHQ23} we described the mixed-angulation and dual S-graph of a generic quadratic differential of infinite area, generalizing the ``WKB triangulation'' in the case of simple zeros.
We can consider each edge, $e$, of the S-graph as an object, $X_e\in \mathcal F_{\mathrm{len}}(S,M,\nu)$ in the Fukaya category.
There is a canonical choice of grading (branch of $\mathrm{Arg}$) such that $\mathrm{Arg}(\sqrt{\varphi})|_e\in(0,\pi)$.

\begin{definition}
    Let $\mathcal F_{\mathrm{core}}(S,M,\nu)\subset \mathcal F_{\mathrm{len}}(S,M,\nu)$ be the full triangulated subcategory generated by the edges of an S-graph.
\end{definition}

If $\varphi$ has no second order pole, then $\mathcal F_{\mathrm{core}}(S,M,\nu)=\mathcal F_{\mathrm{len}}(S,M,\nu)$, i.e.~edges generate the entire category, see \cite[Proposition 6.2]{HKK17}. If $\varphi$ has at least one second order pole, then $\mathcal F_{\mathrm{core}}(S,M,\nu)$ is an orthogonal summand of $\mathcal F_{\mathrm{len}}(S,M,\nu)$ whose orthogonal complement is a direct sum of copies of the category $D^b(\mathrm{Rep}(\mathbf k\mathbb Z))$, the bounded derived category of finite-dimensional representations of the group $\mathbb Z$ over $\mathbf k$, with one such copy for each second order pole.
This follows from the classification of objects in $\mathcal F(S,M,\nu)$, see~\cite[Theorem 4.3]{HKK17}, where an orthogonal summand of the form $D^b(\mathrm{Rep}(\mathbf k\mathbb Z))$ corresponds to objects supported on a simple loop around the second order pole.
Orthogonality follows because a sufficiently small loop around a higher order pole does not intersect any edge of the S-graph, as they have finite length, but the distance to a higher order pole is infinite with respect to the flat metric.
Here we are using the correspondence between intersection points of curves and morphisms in Fukaya categories of surfaces, see e.g.~\cite{IQZ20}.

We note that $\mathcal F_{\mathrm{len}}(S,M,\nu)$ and $\mathcal F_{\mathrm{core}}(S,M,\nu)$ have the same spaces of stability conditions, provided we choose $\A$ as a first homology group as in~\cite{HKK17}, since the objects in the any orthogonal summand corresponding to a double pole are necessarily semistable.

The point of the previous paragraph is that we might as well restrict to $\mathcal F_{\mathrm{core}}(S,M,\nu)$, which does have a classical generator.
The endomorphism algebra of the generator $E\coloneqq \bigoplus_e X_e$ turns out to have a very simple structure: It is formal and determined by a graded quiver with quadratic relations, see the proposition below.
This is shown in \cite[Section 6.1]{HKK17}.

\begin{proposition}
Let $E\coloneqq \bigoplus_e X_e$ be the generator of $\mathcal F_{\mathrm{core}}(S,M,\nu)$ given by the direct sum of all edges of a fixed S-graph $\Sgh$.
Then the $A_\infty$-structure maps $\mathfrak m_d$ of $\mathrm{End}(E)$ vanish for $d\neq 2$, i.e.~$\mathrm{End}(E)$ is essentially a graded algebra.
Moreover, this algebra has the following description in terms of a graded quiver $Q$ with relations:
\begin{enumerate}
    \item Vertices are the edges of $\Sgh$.
    \item Arrows $a_i$ correspond to pairs $(i,i+1)$ of consecutive halfedges (``corners'') of $\Sgh$ with degree given by $d(i,i+1)\in\ZZ$.
    \item Relations are quadratic of the form $a_ja_i=0$ if $i+1\neq j$ are halfedges belonging to the same edge.
\end{enumerate}
\end{proposition}

We assume from now on that $n\geq 1$ and that $M'$ includes all the zeros and simple poles of $\varphi$, i.e.~$M'\subset M$ is maximal.
This is to ensure that we get a proper dg-algebra.

The following theorem describes the endomorphism algebra of the corresponding generator of $\mathcal C_{\mathrm{core}}(S,n)=\mathcal C_{\mathrm{core}}(S,M,\nu,M',n)$ which is by definition the triangulated closure of the image of $\mathcal F_{\mathrm{core}}(S,M,\nu)$ in $\mathcal C(S,M,\nu,M',n)$.
Recall that if there are no second order poles, then $\mathcal F_{\mathrm{core}}(S,M,\nu)=\mathcal F_{\mathrm{len}}(S,M,\nu)$
and thus we also get $\mathcal C_{\mathrm{core}}(S,n)=\mathcal C_{\mathrm{len}}(S,n)$.

\begin{theorem}\label{thm:nonformalgen}
Let $E\coloneqq \bigoplus_e X_e$ be the generator of $\mathcal C_{\mathrm{core}}(S,n)$ given by the direct sum of all edges of a fixed S-graph $\Sgh$.
Then $\mathrm{End}(E)$ is quasi-isomorphic to the finite-dimensional dg-algebra $A(\Sgh,n)$ defined in  Section~\ref{sec_gbga}.
\end{theorem}

\begin{proof}
Let $\mathbb G$ be the collection of arcs which are edges of the S-graph and extend this to an arc system $\mathbb X$ of $(S,M)$.
The curved $A_\infty$-category $\mathcal A_\mathbb X$, whose objects are the arcs in $\mathbb X$, has a full subcategory, denoted $\mathcal A_\mathbb G$, whose objects are the arcs in $\mathbb G$, i.e.~edges of the S-graph.
Since the S-graph is a deformation retract of the surface, the arcs in $\mathbb G$ do not cut out any polygons, and thus there are no contributions to the structure constants of $\mathcal A_\mathbb G$ from immersed $2d$-gons.
In particular, $\mathfrak m_d=0$ for $d\neq 0,2$, so $\mathcal A_\mathbb G$ is essentially a curved algebra.
By construction, $E\in \mathrm{Tw}(\mathcal A_\mathbb G)=\mathcal C_{\mathrm{core}}(S,M,\nu)$ is the twisted complex
\begin{equation}\label{sgraphcomplex}
G(E')=\left(\mathcal A_\mathbb G\oplus \mathcal A_\mathbb G[1-n,n],\begin{pmatrix} 0 & t \\ (-1)^n\mathfrak m_0t^{-1} & 0\end{pmatrix}\right)
\end{equation}
where $E'\in\mathcal F_\mathbb X$ is the object given by the direct sum of edges in $\mathbb G$.
Thus $\mathrm{End}(E)$ is the dg-algebra of 2-by-2 matrices in $\mathcal A_\mathbb G$ with differential given by the commutator with the matrix appearing in \eqref{sgraphcomplex}.

We already know from Proposition~\ref{prop_extcomplex} that the chain complex underlying $\mathrm{End}(E)$ is quasi-isomorphic to the chain complex
\[
\mathcal F_{\mathbb G}[\tau]/\tau^2,\qquad d(a+b\tau)=(-1)^n\mathfrak m_2(b,\mathfrak m_0t^{-1}).
\]
This chain complex has a natural dg-algebra structure, where $\tau$ is given bi-degree $(|\tau|,\pi(\tau))=(n-1,n)\in\mathbb Z\times\mathbb Z/2$ and required to be central in the $\ZZ\times\ZZ/2$-graded sense, i.e.
\[
\tau a=(-1)^{(n-1)|a|+n\pi(a)}a\tau
\]
which forces $\tau^2=0$.
Moreover, it is naturally a sub-dg-algebra of $\mathrm{End}(E)$ via the inclusion map
\[
a+b\tau\mapsto \begin{pmatrix} a & 0 \\ b & a \end{pmatrix}
\]
which is a quasi-isomorphism of dg-algebras.

Given a halfedge $i$ belonging to an edge $e$ of the S-graph, let $\tau_i\coloneqq (-1)^{n\varepsilon(i)}1_e\tau$, where $1_e$ is the idempotent corresponding to $e$ and $\varepsilon(i)$ is $0$ (resp. $1$) if $e$ points away from (resp. towards) $i$.
Because of $|a|+\pi(a)=\varepsilon_-(a)+\varepsilon_+(a)\mod 2$ we then get $a_i\tau_i=(-1)^{|a_i|}\tau_{i+1}a_i$.
Also, $\tau_i=(-1)^n\tau_j$ if $i\neq j$ belong to the same edge.

For any halfedge $i$ define $c_i$ as in \Cref{def_gbga}, then $c_i$ is equal, up to sign, to $c_{p,X}$ from \Cref{def_curvedainfty} where $i$ belongs to the edge $X$ and ends at $p$.
Consider the ideal $I\subset \mathcal F_{\mathbb G}[\tau]/\tau^2$ generated by
\begin{itemize}
\item
$\tau_i$ and $d(\tau_i)=c_i+(-1)^nc_j$ for those halfedges $i$ belonging to an edge $\{i,j\}$ of the S-graph attached to internal vertices only,
\item
$a_i\tau_i$ and $d(a_i\tau_i)=\pm a_{i}c_i$ if $i$ is a halfedge attached to an internal vertex.
\end{itemize}
Since $I$ consists of pairs of basis elements and their images under $d$, it is acyclic and therefore $\mathcal F_{\mathbb G}[\tau]/\tau^2$ is quasi-isomorphic to the quotient dg-algebra $\left(\mathcal F_{\mathbb G}[\tau]/\tau^2\right)/I$, which is precisely $A(\Sgh,n)$.
\end{proof}

\begin{remark}
    As a consequence of the above theorem and the independence of the $A_\infty$-category $\mathcal C_{\mathrm{core}}(S,n)$ on the choice of arc system, we find that the derived Morita equivalence class of $A(\Sgh,n)$ only depends on $S$, $M$, $\nu$, and $n$ ($M'$ was assumed to be maximal).
    In the case of usual (ungraded) Brauer graph algebras, an analogous result was established in~\cite{oz22} likewise using $A_\infty$-structures and arc systems.
\end{remark}

\subsection{Transfer of stability conditions}
\label{subsec_transfer}

In this subsection we recall the main result about stability conditions from~\cite{HKK17} and a theorem from~\cite{Hai21} which, put together, allow us to show that certain components of the space of stability conditions of a full subcategory of $\mathcal C(S,n)$ are identified with moduli spaces of quadratic differentials.

Let us first recall the main result of~\cite{HKK17}.
Let $C$ be a compact Riemann surface and $\varphi$ a quadratic differential which is holomorphic and non-vanishing away from a finite set $D\subset C$ where $\varphi$ is allowed to have zeros, poles and exponential singularities.
Let us discuss how to obtain a triple $(S,M,\nu)$ from $(C,\varphi)$. In the case where $\varphi$ has no exponential singularities, we can simply take $S=C$, $M=D$, and $\nu$ to be the horizontal foliation of $\varphi$.
In general, $S$ will be a compact surface with boundary embedded in $C$, which is the complement of a small open disk around each exponential singularity.
Moreover, for each exponential singularity $z\in C$ with corresponding component $B\subset \partial S$, which is the boundary of the open disk around $z$, and set $E\subset \overline{C\setminus D}$ of infinite-angle conical points (the completion $\overline{C\setminus D}$ is taken with respect to the metric $|\varphi|$), we have a set of non-intersecting paths in $\overline{C\setminus D}$, one from each point in $E$ to a point in $M\cap B$.

If $\varphi$ has no higher order ($\geq 2$) poles, then it is shown in~\cite{HKK17} that there is a stability condition on $\mathcal F(S,M,\nu)$ so that, roughly speaking, $Z$ is given by integrating $\sqrt{\varphi}$ and semistable objects correspond to geodesics on the flat surface $\overline{C\setminus D}$.
In the presence of higher order poles this still works, but one first needs to pass to a full subcategory $\mathcal F_{\mathrm{len}}(S,M,\nu)\subset \mathcal F(S,M,\nu)$ of those objects which are supported away from the higher order poles (so that all central charges, $Z(E)$, are finite).
Here we use the correspondence of objects in $\mathcal F(S,M,\nu)$ and curves on $S$ with extra data.

Instead of fixing $(C,\varphi)$ we can consider the moduli space of such pairs which give rise to equivalent marked surfaces $(S,M,\nu)$.
Thus, as in~\cite{HKK17}, we let $\mathcal M(S,M,\nu)$ be the space of pairs $(C,\varphi)$ corresponding to a marked surface $(S',M',\nu')$ as above, together with a diffeomorphism $f\colon S'\to S$ with $f(M')=M$ and a homotopy class of paths between $f_*\nu'$ and $\nu$ as sections of $\mathbb P(TS)|_{S\setminus M}$.
Passing from $\varphi$ to the cohomology class of $\sqrt{\varphi}$ gives a local homeomorphism
\begin{equation}
\mathcal M(S,M,\nu)   \longrightarrow\mathrm{Hom}(\A_{S,M,\nu},\mathbb C)
\end{equation}
where
\begin{equation}
    \A_{S,M,\nu}\coloneqq H_1(S,M;\mathbb Z\otimes_{\mathbb Z/2}\Sigma)
\end{equation}
and $\Sigma\to S\setminus M$ is the double cover of orientations of $\nu=\mathrm{hor}(\varphi)$ as before.

The main result of \cite{HKK17} is the following.

\begin{theorem}
\label{thm_hkk}
Fix a triple $(S,M,\nu)$. Then there is a canonical map
\begin{equation}
    \mathcal M(S,M,\nu)\longrightarrow \mathrm{Stab}\left(\mathcal F_{\mathrm{len}}(S,M,\nu)\right)
\end{equation}
which is biholomorphic onto a union of connected components.
\end{theorem}

We will combine the above theorem with the following result from~\cite{Hai21}.

\begin{theorem}
\label{thm_transfer}
Let $G\colon \mathcal C\to\mathcal D$ be an exact functor between triangulated categories such that:
\begin{enumerate}
    \item $G$ has a left adjoint $F\colon \mathcal D\to\mathcal C$ and there exists an $n\geq 3$, such that
    \begin{equation}
        \mathrm{Cone}\left(\varepsilon_X\colon FGX\to X\right)\cong X[n]\,,
    \end{equation}
    for any $X\in\mathcal C$, where $\varepsilon$ denotes the counit of the adjunction.
    \item The image of $G$ generates $\mathcal D$ under direct sums, shifts, and cones.
    \item $\mathcal C$ admits a bounded t-structure.
\end{enumerate}
Then $K_0(G)$ induces an isomorphism $K_0(\mathcal C)\to K_0(\mathcal D)$ and there is a pushforward map
\begin{equation}
    G_*\colon \mathrm{Stab}(\mathcal C,\A,\mathrm{cl})\longrightarrow\mathrm{Stab}\left(\mathcal D,\A,\mathrm{cl}\circ \left(K_0(G)\right)^{-1}\right)
\end{equation}
which is a biholomorphism onto its image which is a union of connected components.
Given a stability condition on $\mathcal C$ its image under $G_*$ has the same central charge and $\mathcal D_\phi=G(\mathcal C_\phi)$.
\end{theorem}

We apply the above theorem to the case where $\mathcal C=\mathcal F_{\mathrm{len}}(S,M,\nu)$ and $\mathcal D$ is the full triangulated subcategory, denoted $\mathcal C_{\mathrm{len}}(S,n)$, of $\mathcal C(S,n)$ generated by the image of $\mathcal C$ under the functor $G$ from~\eqref{G_functor}.
If there exist no $p\in M\cap \mathrm{int}(S)$ with $\mathrm{ind}(p)\leq 0$, then $\mathcal C_{\mathrm{len}}(S,n)=\mathcal C(S,n)$.

\begin{theorem}
\label{thm_defstab}
Fix a quintuple $(S,M,\nu,M',n)$ with $n\geq 3$, then there is a canonical map
\begin{equation*}
    \mathcal M(S,M,\nu)\longrightarrow \mathrm{Stab}\left(\mathcal C_{\mathrm{len}}(S,n)\right)
\end{equation*}
which defines a biholomorphism with a union of connected components.
\end{theorem}

\begin{proof}
    Follows from Theorem~\ref{thm_hkk} and Theorem~\ref{thm_transfer}.
    The assumptions of the latter theorem are satisfied by~\eqref{spherical_is_shift}.
\end{proof}

\section{RGB algebras from perverse schobers}\label{subsec:BGAschober}

Consider a weighted marked surface $\sow$ and an integer $n\geq 2$, which we require to be a common multiple of the degrees of the interior singular points of $\sow$. We further choose a mixed angulation of $\sow$ with dual S-graph $\Sgh$. The perverse schober $\psG$ locally arises from finite group quotients of the perverse schobers associated with $n$-gons in~\cite{Chr21b}. Its local sections thus describe finite group quotients of local relative Ginzburg algebras. 

The two main results of this section are as follows: Firstly, we describe the $\infty$-category of global sections of $\psG$ as the derived $\infty$-category of the Koszul dual $\GSn$ of the RGB algebra, see \Cref{thm:BGAschober}. This Koszul dual $\GSn$ further describes a global finite group quotient of a (slightly generalized) relative Ginzburg algebra associated with an $n$-angulated surface, see \Cref{pp:covering}. Secondly, we show that $\psG$ admits a positve arc-system kit, see \Cref{prop:BGAschoberarcsystem}, so that the main result of \cite{CHQ23} applies to describe its space of stability conditions.

The construction of the perverse schober $\psG$ splits into four steps. The final, fourth step consists of gluing together local perverse schobers on the polygons of the mixed-angulation. The first three steps perform the construction of these local perverse schobers in increasing generality: we first focus on the case of degree $n$ vertices of $\Sgh$ of (maximal) valency $n$ and secondly on the case of of vertices of $\Sgh$ of finite degree $m\leq n$ and of (maximal) valency $m$. The passage from the first case to the second amounts to passing to a $\ZZ/\frac{n}{m}$-quotient, making use of the rotational symmetry of the local perverse schober on the $n$-spider/the relative Ginzburg algebra associated with an $n$-gon. Thirdly, we explain how to deal with vertices of $\Sgh$ of arbitrary weight and valency, including the weight $\infty$ case.\\

Before that, we recall some details from the Morita theory for dg-categories.

\subsection{Morita theory for dg-categories}
We denote by $\on{dgCat}$ the $1$-category of $k$-linear dg-categories and dg-functors. We denote by $\on{dgAlg}\subset \on{dgCat}$ the subcategory spanned by dg-categories with a single object, these are called dg-algebras. The $1$-category $\on{dgCat}$ admits the quasi-equivalence model structure, whose weak equivalences $W$ are the dg-functors which induce an equivalence on the level of homotopy categories and quasi-isomorphisms on the morphism chain complexes. We denote the $\infty$-category underlying this model category by $\on{dgCat}[W^{-1}]$. Note that colimits in $\on{dgCat}[W^{-1}]$ can be computed as homotopy colimits in $\on{dgCat}$, see \cite[Section 7.9]{Cis}. There is a colimit preserving functor of $\infty$-categories
\[ \mathcal{D}(\mhyphen):\on{dgCat}[W^{-1}]\rightarrow \on{LinCat}_k\,,\]
mapping a dg-category to its derived $\infty$-category of right dg-modules, see for instance \cite[Section~2.5]{Chr22}.
Note that for a dg-algebra $A$, the image $\mathcal{D}(A)$ is the $\infty$-categorical version of the usual unbounded derived category of the dg-algebra, which justifies the notation $\mathcal{D}(\mhyphen)$. Two dg-categories are called Morita equivalent, if their dg-categories of perfect modules are quasi-equivalent. The functor $\mathcal{D}(\mhyphen)$ sends Morita equivalences to equivalences of $\infty$-categories. Given a dg-category $C$ with finitely many objects $x_1,\dots,x_m$, there is a Morita equivalent dg-algebra $C^{\on{alg}}=\bigoplus_{1\leq i,j\leq m}C(x_i,x_j)$, arising as the sum of all morphism chain complexes in $C$.

We say that a $(\mathbb{Z}$-)graded $k$-linear $1$-category $C$ with finitely many objects is the path category of a ($\mathbb{Z}$-)graded quiver $Q$, if the set of objects of $C$ is given by the set of vertices of $Q$ and the morphisms in $C$ are freely generated over $k$ by the (allowed) composites of the graded arrows of $Q$. In this case, to lift $C$ to a dg-category $C$, it suffices to specify the differentials of the generators, which are the arrows of $Q$. The corresponding Morita equivalent dg-algebra $C^{\on{alg}}$ is given by the path algebra of the quiver $Q$ with differential determined on generators as in $C$.

Given a dg-category $C$, we denote by $\on{dgMod}(C)$ the dg-category of right dg-modules. The action of $C$ on a right dg-module $M$ satisfies the Koszul sign rule
\[ d_M(m.c)=d_M(m).c+(-1)^{\on{deg}(m)}m.d_C(c)\]
for all $c\in C$, $m\in M$.

\subsection{Conventions on perverse schobers}

We briefly state the definition of a parametrized perverse schober. We refer to \cite[Section 3]{CHQ23} and \cite[Sections 3,4]{Chr22} for more background.

For $n\in\mathbb{N}_{\geq 1}$, we let $\rgraph_{n}$ be the ribbon graph with a single vertex $v$ and $n$ incident external edges. We also call $\rgraph_n$ the $n$-spider.  Let $R$ be an $\mathbb{E}_\infty$-ring spectrum, for example $R=k$ a commutative ring. An $R$-linear $\infty$-category is a module over $\on{LMod}_R$ in the $\infty$-category $\mathcal{P}r^L_{\on{St}}$ of stable, presentable $\infty$-categories. We denote the $\infty$-category of $R$-linear $\infty$-categories by $\on{LinCat}_R$.

\begin{definition}\label{def:schobernspider}
Let $n\geq 1$. An $R$-linear perverse schober parametrized by the $n$-spider, or on the $n$-spider for short, consists of the following data:
\begin{enumerate}
\item[(1)] If $n=1$, an $R$-linear spherical adjunction
\[ F\colon \mathcal{V}\longleftrightarrow \mathcal{N}\cocolon G\,,\]
i.e.~an adjunction whose twist functor $T_{\mathcal{V}}=\on{cof}(\on{id}_{\mathcal{V}}\xrightarrow{\on{unit}}GF)\in \on{Fun}(\mathcal{V},\mathcal{V})$ and cotwist functor $T_{\mathcal{N}}=\on{fib}(FG\xrightarrow{\on{counit}}\on{id}_{\mathcal{N}})\in \on{Fun}(\mathcal{N},\mathcal{N})$ are equivalences. The functors $F,G$ are also called spherical functors \cite{AL17}.
\item[(2)] If $n\geq 2$, a collection of $R$-linear adjunctions
\[ (F_i\colon \mathcal{V}^n\longleftrightarrow \mathcal{N}_i\cocolon G_i)_{i\in \mathbb{Z}/n\mathbb{Z}}\]
satisfying that
\begin{enumerate}
    \item $G_i$ is fully faithful, i.e.~$F_iG_i\simeq \on{id}_{\mathcal{N}_i}$ via the counit,
    \item $F_{i}\circ G_{i+1}$ is an equivalence of $\infty$-categories,
    \item $F_i\circ G_j\simeq 0$ if $j\neq i,i+1$,
    \item $G_i$ admits a right adjoint $\on{radj}(G_i)$ and $F_i$ admits a left adjoint $\on{ladj}(F_i)$ and
    \item $\on{fib}(\on{radj}(G_{i+1}))=\on{fib}(F_{i})$ as full subcategories of $\mathcal{V}^n$.
\end{enumerate}
\end{enumerate}
Given a collection of functors $(F_i\colon \mathcal{V}^n\rightarrow \mathcal{N}_i)_{i\in \mathbb{Z}/n}$, we will consider it as a perverse schober on the $n$-spider if there exist adjunctions $(F_i\dashv \on{radj}(F_i))_{i\in \mathbb{Z}/n\mathbb{Z}}$ which define a perverse schober on the $n$-spider.
\end{definition}

Given a ribbon graph $\rgraph$, we denote by $\on{Exit}(\rgraph)$ its exit-path category, whose objects are the vertices and edges of $\rgraph$ and non-identity morphisms go form the vertices to the edges according to incidence. Given a collection of functors $(F_i)_{i\in \mathbb{Z}/n}$ as in \Cref{def:schobernspider}, we can equivalently encode these (without their adjoints) as a functor $\on{Exit}(\rgraph_n)\to \on{LinCat}_R$ as follows: we choose a total order on the cyclically ordered $n$ edges incident to $v$, labeling them by $e_1,\dots,e_n$ accordingly. The functor $\on{Exit}(\rgraph_n)\to \on{LinCat}_R$ is then defined by mapping the morphism $v\to e_i$ to the functor $F_i$. 

\begin{definition}\label{def:schober}
Let $\rgraph$ be a ribbon graph. A functor $\mathcal{F}\colon \on{Exit}(\rgraph)\rightarrow \on{LinCat}_R$ is called an $R$-linear $\rgraph$-parametrized perverse schober if for each vertex $v$ of $\rgraph$, the restriction of $\mathcal{F}$ to $\on{Exit}(\rgraph)_{v/}$ determines a perverse schober parametrized by the $n$-spider in the sense of \Cref{def:schobernspider}.
\end{definition}

It is of course also possible to consider perverse schobers valued the $\infty$-category of all stable $\infty$-categories, as in \Cref{cor:intro}. We restrict in the main text to linear and presentable $\infty$-categories for technical convenience. 

\begin{definition}
Let $\mathcal{F}\colon \on{Exit}(\rgraph)\rightarrow \on{LinCat}_R$ be an $R$-linear $\rgraph$-parametrized perverse schober. The $\infty$-category $\glsec(\rgraph,\mathcal{F})\in \on{LinCat}_R$ of global sections of $\mathcal{F}$ is defined as the limit of $\mathcal{F}$.
\end{definition}

\subsection{The construction of the perverse schober}\hfill\par
\medskip
\noindent  \underline{\bf Step 1:} Definition for the $n$-spider with vertex of degree $n$.

\begin{definition}\label{def:Gn1}
Let $n\geq 2$.
\begin{enumerate}
    \item[(1)]We denote by $Q_{n,n}$ the graded quiver with vertices $1,\dots,n$ and arrows $\alpha_{i,j}\colon i\rightarrow j$ for all $1\leq i,j \leq n$, $i\neq j$, as well as loops $l_i,L_i\colon i\to i$ for all $1\leq i\leq n$. The degree of $\alpha_{i,j}$ is given by
    \[ \on{deg}(\alpha_{i,j})=\begin{cases} j-i+1 & \text{if } j<i\,,\\ j-i+1-n & \text{if }j> i\,.\end{cases}\]
    The degrees of $L_i$ and $l_i$ are given by
    \[ \on{deg}(L_i)=\on{deg}(l_i)-1=1-n\,.\]
    \item[(2)] We let $G_{n,n}$, be the dg-category arising from the graded quiver $Q_{n,n}$, with the differentials determined on the generators by
\[ d(\alpha_{i,j})= \begin{cases} \sum_{j\leq k\leq i}(-1)^{j-k-1} \alpha_{k,j}\alpha_{i,k}& \text{if }j<i\\
\sum_{1\leq k \leq i}(-1)^{j-k+n-1} \alpha_{k,j}\alpha_{i,k}+ \sum_{j\leq k\leq n}(-1)^{j-k-1} \alpha_{k,j}\alpha_{i,k}& \text{if }j>i
\,,\end{cases}\]
\[ d(l_i)=0\,,\]
and
\[ d(L_i)=-l_i+\sum_{j<i}(-1)^{i-j+1-n} \alpha_{j,i}\alpha_{i,j}+\sum_{j>i}(-1)^{i-j-1} \alpha_{j,i}\alpha_{i,j}\,.\]
\item[(3)] Let $k[t_{2-n}]$ be the graded polynomial algebra with $|t_{2-n}|=2-n$.
We specify the dg-functor
\[ \psi_{n,n}\colon k[t_{2-n}]^{\amalg n}\longrightarrow G_{n,n}\,,\]
by mapping the generator $t_{2-n}$ in the $i$-th component $k[t_{2-n}]$ to $l_i$. It is straightforward to check that this is a cofibration between cofibrant dg-categories, with respect to the quasi-equivalence model structure on $\on{dgCat}$, by verifying the left lifting property with respect to acyclic fibrations.
\end{enumerate}
\end{definition}

For $1\leq i\leq n$, we denote by $\pi_i\colon \mathcal{D}(k[t_{2-n}]^{\amalg n})\simeq \mathcal{D}(k[t_{2-n}])^{\oplus n}\rightarrow \mathcal{D}(k[t_{2-n}])$ the projection to the $i$-th direct summand and by $\iota_i$ the two-sided adjoint of $\pi_i$, given by the inclusion of the $i$-th direct summand.

\begin{proposition}\label{prop:quotschobm=1}
The collection of adjunctions
\[ ((\psi_{n,n})_!\circ \iota_i\colon \mathcal{D}(k[t_{2-n}])\longleftrightarrow \mathcal{D}(G_{n,n})\cocolon \pi_i\circ \psi_{n,n}^*)_{1\leq i\leq n}\]
defines a perverse schober on the $n$-spider.
\end{proposition}

\begin{proof}
To prove that these adjunctions form a perverse schober, we check that they describe (up to composition with an involutive autoequivalence) the local model of a perverse schober of \cite[Prop.~3.7]{CHQ23}. This description in terms of the local model is the content of \cite[Prop.~4.25]{Chr21b}, using that $G_{n,n}$ is Morita equivalent to the dg-category denoted $D_n$ in loc.~cit. The dg-category $D_n$ is obtained from $G_{n,n}$ by simply discarding the loops $L_i$ and $l_i$, as in \Cref{rem:reducedGinzburg}.
\end{proof}

The dg-category $G_{n,n}$ is Morita equivalent to the relative (higher) Ginzburg algebra associated with an $n$-gon, as considered in \cite{Chr21b}. The advantage of considering $G_{n,n}$ over this dg-algebra are the cofibrancy properties mentioned in \Cref{def:Gn1}.

Its rotational symmetry provides $G_{n,n}$ with a group action of $\ZZ/n$, i.e.~a functor $\xi_n\colon B\ZZ/n\rightarrow \on{dgCat}$ from the classifying space of $\ZZ/n$ into the the $1$-category of dg-categories, mapping $\ast\in B\ZZ/n$ to $G_{n,n}$. For any factor $m$ of $n$, we can restrict the group action to $\xi_{n,m}\colon B\ZZ/{\frac{n}{m}}\rightarrow \on{dgCat}$. We denote by $G_{n,n}^{\on{alg}}\in \on{dgAlg}$ the dg-algebra of morphisms in $G_{n,n}$, which is Morita equivalent to $G_{n,n}$. The group action $\xi_{n,m}$ induces a group action $\xi_{n,m}^{\on{alg}}\colon B\ZZ/{\frac{n}{m}}\rightarrow \on{dgAlg}$, mapping $\ast \in B\ZZ/{\frac{n}{m}}$ to $G_{n,n}^{\on{alg}}$. In the next step, we construct a perverse schober on the $m$-spider from the derived $\infty$-category of the dg-algebra of fixed points of $\xi_{n,m}^{\on{alg}}$.

\medskip
\noindent \underline{\bf Step 2:} Construction for the $m$-spider with vertex of degree $m$ dividing $n$.

\begin{definition}
The dg-algebra $G_{n,m}^{\on{alg}}$ is defined as the fixed points (also called invariants)
\[ G_{n,m}^{\on{alg}}\coloneqq \lim \xi_{n,m}^{\on{alg}}\in \on{dgAlg}\,.\]
\end{definition}

We emphasize that the above limit is taken in the $1$-category of dg-algebras. We note that this limit is also equivalent to the colimit in the $1$-category of dg-algebras, i.e.~to the quotient dg-algebra. Concretely, $G_{n,m}^{\on{alg}}$ is the dg-algebra arising from the graded quotient quiver $Q_{n,m}=Q_{n,n}/\ZZ/{\frac{n}{m}}$,
(with the quotient differential).

We denote by $G_{n,m}$ the dg-category arising from the graded quiver $Q_{n,m}$, such that the differentials are as in $G_{n,m}^{\on{alg}}$. We will sometimes consider the paths in $Q_{n,n}$ via the quotient map on quivers as morphisms in $G_{n,m}$.

We proceed by showing that $G_{n,m}^{\on{alg}}$ also describes the homotopy fixed points of $\xi_{n,m}^{\on{alg}}$, with respect to the quasi-isomorphism model structure on $\on{dgAlg}$ (this model structure is defined for instance in \cite[7.1.4.5]{HA}).

\begin{lemma}\label{lem:BZholim}
Suppose that $\on{char}(k)\neq a$. Let $C$ be a $k$-linear model category such that $\on{Fun}(B\ZZ/a,\C)$ admits the injective model structure. Consider a diagram \[\xi\colon B\ZZ/a\rightarrow C\] valued in fibrant objects and fibrations. Then any limit cone of $\xi$ is also a homotopy limit cone.
\end{lemma}

\begin{proof}
It suffices to show that the diagram $\xi$ defines a fibrant object with respect to the injective model structure on the $1$-category of functors $\on{Fun}(B\ZZ/{a},C)$, whose weak equivalences and cofibrations are the pointwise weak equivalences and cofibrations, respectively. For this, we verify the right lifting property with respect to acyclic cofibration. Let $\alpha\rightarrow \beta$ be an acyclic cofibration in $\on{Fun}(B\ZZ/{a},C)$ together with a natural transformation $\alpha\to \xi$. Evaluating at the unique object $\ast\in B\ZZ/{a}$, we can find a lift in $C$ \[
\begin{tikzcd}
\alpha(\ast) \arrow[d] \arrow[r]         & \xi(\ast) \\
\beta(\ast) \arrow[ru, "\kappa", dashed] &
\end{tikzcd}
\]
since $\xi(\ast)$ is fibrant. Let $q\colon \ast\rightarrow \ast\in \ZZ/{a}$ be the endomorphism corresponding to the generator $1\in \ZZ/{a}$. Using that $\on{char}(k)\neq a$, we can define $\kappa'=\frac{1}{a}\sum_{i=0}^{a-1} \xi(q)^{a-i}\kappa \beta(q)^i$, satisfying $\kappa'\circ \beta(q)=\xi(q)\circ \kappa'$. The map $\kappa'$ hence defines the desired lift in $\on{Fun}(\on{B}\ZZ/{a},C)$
\[
\begin{tikzcd}
\alpha \arrow[d] \arrow[r]          & \xi\,, \\
\beta \arrow[ru, "\kappa'", dashed] &
\end{tikzcd}
\]
concluding the proof.
\end{proof}

\begin{corollary}\label{cor:Gnmhtplim}
Suppose that $\on{char}(k)\neq \frac{n}{m}$. Then $G_{n,m}^{\on{alg}}$ is both the homotopy limit and homotopy colimit of $\xi_{n,m}^{\on{alg}}$ with respect to the quasi-isomorphism model structure on $\on{dgAlg}$.
\end{corollary}

\begin{proof}
This follows from \Cref{lem:BZholim} since $G_{n,n}$ is fibrant and cofibrant, and any isomorphism is a (co)fibration.
\end{proof}

\begin{lemma}\label{lem:homologyGnm}~
Suppose that $\on{char}(k)\neq \frac{n}{m}$.
\begin{enumerate}
    \item[(1)] The cohomology of $G_{n,m}^{\on{alg}}$ is equivalent to $k[t_{2-n}]^{\oplus 2 m}$ and generated over $k$ by \[ \{l_i^q,l_{i}^q\alpha_{i+1,i}\}_{q\geq 0,\,1\leq i\leq m}\,.\]
    \item[(2)] The path $\alpha_{i+1,i}l_{i+1}$ is cohomologous to $(-1)^{n}l_{i}\alpha_{i+1,i}$.
\end{enumerate}
\end{lemma}

\begin{proof}
We only prove part (1), part (2) follows from a direct computation.

The homotopy limit $G_{n,m}^{\on{alg}}$ of $\xi_{n,m}^{\on{alg}}$ describes the $\infty$-categorical limit of $\xi_{n,m}^{\on{alg}}$, considered as a diagram valued in the $\infty$-category $\on{Alg}^{\on{dg}}(k)$ underlying the model category $\on{dgAlg}$. By \cite[7.1.4.6, 3.2.2.5]{HA}, we find that taking the $\infty$-categorical limit commutes with the forgetful functor $\on{Alg}^{\on{dg}}(k)\rightarrow \mathcal{D}(k)$. The cohomology of $G_{n,n}^{\on{alg}}$ was shown in \cite[Lemma 4.27]{Chr21b} to be given by $k[t_{2-n}]^{\oplus 2n}$. The group $\ZZ/\frac{n}{m}$ acts on $k[t_{2-n}]^{\oplus 2n}\in \mathcal{D}(k)$ by cyclically permuting summands, its homotopy fixed points are hence given by $k[t_{2-n}]^{\oplus 2 m}$, as desired. The description of the generators of the cohomology follows from the corresponding description of the generators of $\on{H}^\ast(G_{n,n}^{\on{alg}})$ in the proof of Lemma 4.27 in \cite{Chr21b}.
\end{proof}

Consider the dg-functor
\[  \psi_{n,m}\colon k[t_{2-n}]^{\amalg m}\to G_{n,m}\,,\]
mapping the $i$-th generator $t_{2-n}$ to the loop $l_i$.

\begin{proposition}\label{prop:sphGnm}
Suppose that $\on{char}(k)\neq  \frac{n}{m}$. The adjunction
\[(\psi_{n,m})_!\colon \mathcal{D}(k[t_{2-n}]^{\amalg m})\longleftrightarrow \mathcal{D}(G_{n,m})\cocolon \psi_{n,m}^*\]
is spherical.
\end{proposition}

\begin{proof}
Consider the functor
\[ (\psi_{n,m}^{\on{alg}})_!\colon \mathcal{D}(k[t_{2-n}]^{\oplus m}) \simeq \mathcal{D}(k[t_{2-n}]^{\amalg m}) \xrightarrow{(\psi_{n,m})_!} \mathcal{D}(G_{n,m})\simeq \mathcal{D}(G_{n,m}^{\on{alg}})\,.\] This functor arises from the morphism of dg-algebras
\[ \psi_{n,m}^{\on{alg}}\colon k[t_{n-2}]^{\oplus m}\xlongrightarrow{\oplus_{i}(t_{2-n}\mapsto l_i)} G_{n,m}^{\on{alg}}\,.\]
We can equivalently prove that $(\psi_{n,m}^{\on{alg}})_!\dashv (\psi_{n,m}^{\on{alg}})^*$ is a spherical adjunction.

\Cref{lem:homologyGnm} shows that $(\psi_{n,m}^{\on{alg}})^*(\psi_{n,m}^{\on{alg}})_!(k[t_{2-n}]^{\oplus m})\simeq k[t_{2-n}]^{\oplus 2m}$, as a $k[t_{2-n}]^{\oplus m}$-bimodule. The left and right actions of $t_{2-n}$ lying in the $i$-th direct summand on the generators $\{l_j^q,l_j^q\alpha_{j,j+1}\}_{q\geq 0,\,1\leq j\leq m}$ are given by
\begin{align*} t_{2-n}l_j^q=l_j^q.t_{2-n}&=\begin{cases} l_i^{q+1} & i=j\\ 0 & i\neq j\end{cases}\\
t_{2-n}.l_j^q\alpha_{j+1,j}& =\begin{cases} l_i^{q+1}\alpha_{i+1,i} & i=j\\ 0 & i\neq j\end{cases}\\
l_j^q\alpha_{j+1,j}.t_{2-n}& =\begin{cases} (-1)^nl_j^{q+1}\alpha_{j+1,j} & i=j+1 \on{mod} n\\ 0 & i\neq j+1\on{mod} n\end{cases}
\end{align*}
The twist functor of the adjunction $(\psi_{n,m}^{\on{alg}})_!\dashv (\psi_{n,m}^{\on{alg}})^*$ is hence given by an autoequivalence of $\mathcal{D}(k[t_{2-n}])^{\oplus m}\simeq \mathcal{D}(k[t_{2-n}]^{\oplus m})$ which permutes the components cyclically by one step and acts on each component via the involution $\phi_!$ arising from the dg-algebra morphism
\begin{equation}\label{eq:phiinvolution}
\phi:k[t_{2-n}]\xrightarrow{t_{2-n} \mapsto (-1)^nt_{2-n}} k[t_{2-n}]\,.
\end{equation}

A similar compution shows that the cotwist functor is also invertible, giving us the desired sphericalness of the adjunction.
\end{proof}

For $1\leq i\leq m$, we denote by $\pi_i\colon \mathcal{D}(k[t_{2-n}]^{\amalg m})\simeq \mathcal{D}(k[t_{2-n}])^{\oplus m}\leftrightarrow \mathcal{D}(k[t_{2-n}])\noloc \iota_i$ the adjunction from above. We arrive at the generalization of \Cref{prop:quotschobm=1}:

\begin{proposition}\label{prop:quotschob}
Suppose that $\on{char}(k)\neq \frac{n}{m}$. The collection of functors
\[ (\pi_i\circ \psi_{n,m}^*\colon \mathcal{D}(G_{n,m})\longrightarrow  \mathcal{D}(k[t_{2-n}]))_{1\leq i\leq m} \]
define a perverse schober on the $m$-spider.
\end{proposition}

\begin{proof}
If $m=1$, $\psi_{n,m}^*$ is by \Cref{prop:sphGnm} a spherical functor, which amounts to a perverse schober on the $1$-spider. Suppose that $m\neq 1$. It follows from the description of the twist functor of the adjunction $(\psi_{n,m})_!\dashv \psi_{n,m}^*$ in the proof of \Cref{prop:sphGnm} and \cite[Cor.~2.5.16]{DKSS19} that $\pi_{i}\circ \psi_{n,m}^*$ is, after composition with the autoequivalence $\phi_!$ of $\mathcal{D}(k[t_{2-n}])$ from \eqref{eq:phiinvolution}, left adjoint to $(\psi_{n,m})_!\circ \iota_{i-1}$. These adjunctions, together with the description of the functor $\psi_{n,m}^*(\psi_{n,m})_!$ in the proof of \Cref{prop:sphGnm}, directly imply all conditions of \Cref{def:schobernspider}.
\end{proof}

\medskip
\noindent \underline{\bf Step 3:} Construction for the $r$-spider with vertex $v$ of degree $m\geq r$, with $m$ dividing $n$.

\begin{construction}\label{constr:Fnmr}
Given an $l\in \mathbb{N}$, we denote by $\Sgh_{l}$ the ribbon graph with a single vertex $v$ and $l$ incident external edges. Let $1\leq r\leq m$. Note that the following two types of data are equivalent\footnote{One passes from the data as in i) to data as in ii) by setting $d(a,b)=l$ for two consecutive halfedges $a,b$ of $\spider_{k}$ if $b$ follows in $\spider_j$ after the halfedge $a$ after $l$ steps.}:
\begin{enumerate}
\item[i)] an identification of the halfedges of $\Sgh_{r}$ with a subset of the halfedges of $\Sgh_{m}$, respecting the total orders of the halfedges.
\item[ii)] integers $d(a,b)\geq 1$ for all consecutive halfedges of $\Sgh_{r}$ at $v$, such that the sum of all these integers is equal to $m$.
\end{enumerate}
The latter type of data arises if $\Sgh_r$ is the subgraph of an S-graph consisting of an $r$-valent interior vertex of degree $m$ and its incident halfedges.

We choose a total order of the halfedges of $\Sgh_m$, compatible with their cyclic order. With data as in i) above, this also determines a total order of the halfedges of $\Sgh_r$.
Consider the $\Sgh_{m}$-parametrized perverse schober from \Cref{prop:quotschob}, in the following denoted $\hF_{n,m}$. We can obtain from $\hF_{n,m}$ an $\Sgh_{r}$-parametrized perverse schober, denoted $\hF_v^n$, by replacing $\hF_{n,m}(v)$ with the kernel of all functors $\hF_{n,m}(v\xrightarrow{c}e_c)$, with $c\in e_c$ a halfedge of $\Sgh_{m}$ which is not a halfedge of $\Sgh_{r}$, see \cite[Prop.~3.6]{CHQ23}. We remark that the notation $\hF_v^n$ omits the dependence on the 'S-graph data' ii) above at the vertex $v$.

Concretely, $\hF_v^n$ is equivalent to the derived category of the dg-category, denoted $G_{v}^n$, with
\begin{itemize}
\item objects the halfedges of $\Sgh_{r}$, considered in the following as a subset of halfedges $1,\dots,m$ of $\Sgh_{m}$.
\item morphisms freely generated by the arrows of the graded subquiver $Q_v^n$ of $Q_{n,m}$, whose vertices are the halfedges of $\Sgh_r$, and whose arrows consist of those which both begin and end at halfedges of $\Sgh_r$.
\item differentials on generators defined as in $G_{n,m}$, but setting to zero all paths containing an arrow which does not both begin and end at halfedges of $\Sgh_r$.
\end{itemize}

The dg-category $G_v^n$ also describes a homotopy pushout of $G_{n,m}$ along the dg-functor $k[t_{n-2}]^{\amalg m-r}\to 0^{\amalg  m-r}$. We note that the notation $G_{v}^n$ again leaves choices of data as in i) or ii) implicit. Let $1\leq i\leq r$ and $j$ be such that the $i$-th halfedge of $S_r$ is identified with the $j$-th halfedge of $S_m$. The functor $\hF_v^n(v\rightarrow i)$ is equivalent to the pullback along the dg-functor (i.e.~the right adjoint of the image under $\mathcal{D}(\mhyphen)$)
\begin{equation}\label{eq:ci'} k[t_{2-n}]\xrightarrow{t_{2-n}\mapsto l_j} G_v^n\,.\end{equation}
\end{construction}

\begin{construction}\label{constr:Fnr}
Let $r\geq 1$ and consider the ribbon graph $\Sgh_r$ consisting of a single vertex $v$ and $r$ incident external edges $e_1,\dots,e_r$. Suppose that the edges of $\Sgh_r$ are equipped with a total order, compatible with their given cyclic order. Below, graphs of this form describe the part of an extended ribbon graph of an S-graph lying at a boundary vertex $v$ of valency $r-1$. From this perspective, the terminal edge in the total order corresponds to the virtual halfedge in the sense of \Cref{def:extdgraph}. We thus ask that each pair of subsequent edges $i,i+1$ is equipped with an integer $d(i,{i+1})\geq 1$, where ${i+1}\neq r$ is not the terminal edge in the total order. We denote $d(j,i)=d(j,{j+1})+d({j+1},{j+2})+\dots+d({i-1},i)$ for all $j<i$.

Consider the spherical functor $0\colon 0\to \mathcal{D}(k[t_{2-n}])$. There is a corresponding $\Sgh_r$-parametrized perverse schober, denoted $\hF_r(0)$, given by the collection of functors
\[\left(\varrho_i\colon \on{Fun}(\Delta^{r-2},\mathcal{D}(k[t_{2-n}]))\longrightarrow  \mathcal{D}(k[t_{2-n}])\right)_{i\in \mathbb{Z}/r}\]
with $\varrho_1$ the pullback along the inclusion $\Delta^{0}\simeq \Delta^{\{r-2\}}\subset \Delta^{r-2}$ of the final vertex into the $n$-simplex and $\varrho_i$ the $(2i-2)$-fold left adjoint of $\varrho_1$, see \cite[Prop.~3.7]{CHQ23}. 
We modify $\hF_r(0)$ by composing $\hF_r(0)(v\to e_{i})$ with the autoequivalence $\phi_!^{i+1-r}[i-1-d(1,i)]$ for all $1\leq i \leq r-1$, setting $d(1,1)=0$ and with $\phi$ the involution from \eqref{eq:phiinvolution}. We denote, again abusively, the resulting $\Sgh_r$-parametrized perverse schober by $\hF_v^n$.
\end{construction}

\begin{lemma}\label{lem:Gnr} Consider the setup of \Cref{constr:Fnr} (and recall that $v$ should be thought of as an $(r-1)$-valent boundary vertex of an S-graph). The $\infty$-category of global sections $\glsec(\Sgh_r,\hF_v^n)$ is equivalent to the derived $\infty$-category of the dg-category, denoted $G_{v}^n$, arising from the graded quiver $Q_{v}^n$ with vertices $1,\dots,r-1$ and arrows
\begin{itemize}
    \item $\alpha_{i,j}\colon i\rightarrow j$ in degree $1-d(j,i)$ for all $1\leq j<i\leq r-1$,
    \item $\beta_{i,j}\colon i\rightarrow j$ in degree $2-n-d(j,i)$ for all $1\leq j\leq i\leq r-1$, Note that $i=j$ is allowed, with $d(i,i)=0$,
\end{itemize}
and with differential determined on generators by
\[ d(\alpha_{i,j})=\sum_{j<k<i} (-1)^{d(j,k)-1} \alpha_{k,j}\alpha_{i,k}\]
and
\[ d(\beta_{i,j})=\sum_{j< k \leq i}-\alpha_{k,j}\beta_{i,k} + \sum_{j\leq k <i} (-1)^{n+d(j,k)} \beta_{k,j}\alpha_{i,k}\,.\]
\end{lemma}

\begin{proof}
Let $B_v^n$ be the dg-category with objects $1,\dots,r-1$, morphism complexes generated by the graded morphisms
\begin{itemize}
    \item $\alpha_{i+1,i}\colon i+1\rightarrow i$ in degree $1-d(i,i+1)$ for all $1\leq i\leq r-2$,
    \item $\beta_{i,i}\colon i\rightarrow i$ in degree $2-n$ for all $1\leq i\leq r-1$,
\end{itemize}
subject to the relations $\alpha_{i+1,i}\alpha_{i+2,i+1}=0$ and $\beta_{i,i}\alpha_{i+1,i}=\alpha_{i+1,i}\beta_{i+1,i+1}$. All differentials of morphisms in $B_v^n$ are understood to vanish. The dg-category $B_v^n$ is Morita equivalent to the upper triangular dg-algebra $(B_v^{n})^{\on{alg}}$
\[
\begin{pmatrix}
k[t_{2-n}] & k[t_{2-n}][d({r-2},{r-1})-1] &  \dots & 0 & 0 & 0 \\
0 & k[t_{2-n}] & \dots & 0 &0 &0\\
\vdots & \vdots & \ddots& \vdots& \vdots& \vdots\\
0 & 0 & \dots  & k[t_{2-n}] & k[t_{2-n}][d(2,3)-1] & 0 \\
0 & 0  & \dots& 0 & k[t_{2-n}] & k[t_{2-n}][d(1,2)-1] \\
0 & 0  & \dots& 0 & 0& k[t_{2-n}]
\end{pmatrix}
\]
see also \cite[Def.~2.36]{Chr22} for what we mean by upper-triangular dg-algebra. By \cite[Prop.~2.39]{Chr22}, there exists an equivalence of $\infty$-categories
\[ \mathcal{D}(B_{n,r}^{\on{alg}})\simeq \hF_v^n(v)\simeq \glsec(\Sgh_r,\hF_v^n)\,.\]

Consider the dg-functor $\pi\colon G_v^n\to B_v^n$, mapping the object $i$ to the object $i$, mapping $\alpha_{i+1,i}$ to $\alpha_{i+1,i}$, $\beta_{i,i}$ to $(-1)^{n(r-1-i)}\beta_{i,i}$ and all other generators to $0$. We prove that $\pi$ is a quasi-equivalence by performing an induction on $r$ and arguing by analyzing the derived Homs in the derived $\infty$-categories. In the cases $r=1,2,3,4$, it follows from a straightforward direct computation that $\pi$ is a quasi-equivalence. We proceed with the induction step. Each object $i$ of $G_v^n$ defines a projective $G_v^n$ dg-module $M_i$, whose underlying chain complex is generated by paths beginning at $i$. Denote by $M_i'\in \mathcal{D}(B_v^n)$ the image of $M_i$ under $\pi$. Let $\rgraph^i$ be the subgraph of $\rgraph_r$ consisting of the edges $1,\dots,i-1,i+1,\dots r$.
Denote the vertex of $\rgraph^i$ by $v^i$.

Note that $G_{v^1}^n,G_{v^{r-1}}^n$ and $B_{v^1}^n,B_{v^{r-1}}^n$ are full dg-subcategories of $G_v^n$ and $B_v^n$. Applying the induction assumption thus yields that
\begin{equation}\label{eq:homeq} \mathcal{D}(\pi)\colon \on{RHom}_{\mathcal{D}(G_v^n)}(M_i,M_j)\simeq \on{RHom}_{\mathcal{D}(B_{v}^n)}(M_{i}',M_{j}')\end{equation}
is an equivalence for all $1\leq i,j\leq r-1$, such that $(i,j)\neq (1,r-1)$. In the next paragraph, we prove that $\on{RHom}_{\mathcal{D}(G_v^n)}(M_1,M_{r-1})\simeq 0$. Using that $G_v^n\simeq \bigoplus_{1\leq i\leq r-1} M_i$ and $B_v^n\simeq \bigoplus_{1\leq i\leq r-1} M_i'$, we then obtain the quasi-isomorphism
\begin{align*} \pi^{\on{alg}}\colon (G_v^n)^{\on{alg}}\simeq & \on{RHom}_{\mathcal{D}(G_v^n)}(G_v^n,G_v^n)\simeq \on{RHom}_{\mathcal{D}(G_v^n)}\left(\bigoplus_{1\leq i\leq r-1} M_i,\bigoplus_{1\leq j\leq r-1}M_j\right)\\
\simeq &\on{RHom}_{\mathcal{D}(B_v^n)}\left(\bigoplus_{1\leq i\leq r-1} M_i',\bigoplus_{1\leq j\leq r-1}M_j'\right) \simeq \on{RHom}_{\mathcal{D}(B_v^n)}(B_v^n,B_v^n)\simeq (B_v^n)^{\on{alg}}\,,\end{align*}
implying that $\pi$ is indeed a quasi-equivalence.

Consider the cofibration $\alpha_i:k[t_{2-n}]\xrightarrow{t_{2-n}\mapsto \beta_{i,i}} G_v^n$. Its cofiber and homotopy cofiber coincide and are given by $G_{v^i}^n$. Applying $\mathcal{D}(\mhyphen)$, we obtain a fully faithful functor $\mathcal{D}(\alpha_i)\simeq (\mhyphen)\otimes_{k[t_{2-n}]}M_i:\mathcal{D}(k[t_{2-n}])\to \mathcal{D}(G_v^n)$ with right adjoint $\on{radj}(\mathcal{D}(\alpha_i))=\on{RHom}_{\mathcal{D}(G_v^n)}(M_i,\mhyphen)$. Choose any $1<i<r-2$, e.g.~$i=2$. As follows from \eqref{eq:homeq}, we have
\[ M_1,M_{r-1}\in \on{fib}(\on{radj}(\mathcal{D}(\alpha_i)))\simeq \on{cof}(\mathcal{D}(\alpha_i))\subset \mathcal{D}(G_v^n)\,.\]
Under the identification $\on{cof}(\mathcal{D}(\alpha_i))\simeq \mathcal{D}(G_{v^i}^n)$, $M_1$ and $M_{r-1}$ are identified with $M_1$ and $M_{r-2}$, respectively. We thus have
\[ \on{RHom}_{{\mathcal{D}(G_v^n)}}(M_1,M_{r-1})\simeq \on{RHom}_{\on{fib}(\on{radj}(\mathcal{D}(\alpha_i)))}(M_1,M_{r-1})\simeq \on{RHom}_{{\mathcal{D}(G_{v^i}^n)}}(M_1,M_{r-2})\simeq 0\,,
\]
as desired.
\end{proof}

We remark that the shifts and composites with $\phi_!$ in \Cref{constr:Fnr} have been set up such that $\hF_v^n(v\rightarrow e_i)$ is given by the pullback along the dg-functor
\begin{equation}\label{eq:betaii} k[t_{2-n}]\rightarrow G_{v}^n,\,t_{2-n}\mapsto \beta_{i,i}\,,\end{equation}
for any $1\leq i\leq r-1$.

\medskip
\noindent \underline{\bf Step 4:} Gluing together the local perverse schobers from step 3.

We fix a weighted marked surface $\sow$ with a mixed-angulation $\AS$ and dual S-graph $\Sgh=\dAS$. We assume that $\on{char}(k)$ is not divided by $\mathfrak{n}{m}$ for each $m$ appearing as the degree of a singular point $x\in \Delta\cap \on{int}({\bf S})$. To correctly incorporate vertices of the S-graph $\Sgh$ lying at boundary singular points into the formalism of perverse schobers, we introduce the following:

\begin{definition}\label{def:extdgraph}
The \textit{extended graph} $\eS$ of the S-graph $\Sgh$ is obtained by adding an external edge to $\Sgh$ at each boundary vertex. This edge is placed at the final position in the total order of the halfedges at the boundary vertex, which induces a compatible cyclic order of the halfedges. Hence, we consider $\eS$ as a ribbon graph. We call these so added external edges, and their halfedges, virtual. The non-virtual edges of $\eS$ can thus be identified with the edges of $\Sgh$.
\end{definition}

\begin{construction}\label{constr:Frgraph}
For each internal vertex of $\Sgh$ of degree $m<\infty$ and valency $r\leq m$, we have by \Cref{constr:Fnmr} a perverse schober on the $r$-spider $\hF_v^n$. Similarly, we have for each boundary vertex of $\Sgh$ with valency $r-1$ by \Cref{constr:Fnr} a perverse schober on the $r$-spider $\hF_{v}^n$.

For each edge $e$ of $\Sgh$, we choose an incident vertex $v$ and compose $\hF_v^n(v\rightarrow e)$ with the involution of $\hF_v^n(e)=\mathcal{D}(k[t_{2-n}])$ arising from the dg-functor $k[t_{2-n}]\xrightarrow{t_{2-n}\mapsto (-1)^n t_{2-n}} k[t_{2-n}]$, to obtain the perverse schober denoted $(\hF_v^n)'$. This step is necessary to ensure correct signs later on.

We define the $\eS$-parametrized perverse schober $\psG$ by gluing these perverse schobers $(\hF_v^n)'$, i.e.~$\psG$ is the unique diagram which restricts to the diagrams $(\hF_v^n)'$ on the subgraphs $\Sgh_r\subset \eS$.
\end{construction}

\subsection{Global sections and the Koszul duals of RGB algebras}

\begin{construction}\label{constr:Grgraph}
For each vertex $v$ of $\Sgh$, we have by \Cref{constr:Fnmr} and \Cref{lem:Gnr} a dg-category $G_{v}^n$, arising from the graded quiver $Q_v^n$. We define the functor $$\underline{G}({\Sgh},n):\on{Exit}(\Sgh)^{\on{op}}\rightarrow \on{dgCat}$$ by mapping
\begin{itemize}
\item each vertex $v$ to the dg-category $G_v^n$,
\item each edge $e$ to the dg-algebra $k[t_{2-n}]$ and
\item each morphism $v\xrightarrow{a} e$, with $v$ a vertex and $a$ an incident halfedge, to the dg-functor $k[t_{2-n}]\to G_v^n$ corresponding to that halfedge described in \eqref{eq:ci'} or \eqref{eq:betaii}. We denote this dg-functor for later use by $\zeta_{v,a}$.
\end{itemize}
We define dg-category $\GSn$ as the ($1$-categorical) colimit of $\underline{G}({\Sgh},n)$. We can describe $\GSn$ as follows:

Let $\QSn$ be the graded quiver with
\begin{itemize}
\item vertices the edges of $\Sgh$,
\item and arrows obtained by including for each vertex $v$ all arrows of $Q_v^n$ in $\QSn$ (via the apparent map from the set of vertices of $Q_v^n$ to the vertices of $\QSn$). For each edge $e$ of $\Sgh$, we have added two loops at $e$ lying in degree $2-n$, arising as the images of $t_{2-n}$ in \eqref{eq:ci'} or \eqref{eq:betaii}. We further identify these two loops in $\QSn$ and denote the resulting loop by $l_e$.
\end{itemize}
The graded category underlying $\GSn$ is the path category of the graded quiver $\QSn$. The differential of a generator $\alpha$ lying in the subset $G_v^n\subset \GSn$ given by the same formula as for $d(\alpha)$ in the dg-category $G_v^n$. Note that for an edge $e$ of $\Sgh$, we have $d(l_e)=0$.
\end{construction}

\begin{remark}\label{rem:reducedGinzburg}
We can reduce the dg-category $\GSn$ to a quasi-equivalent and 'smaller', yet still cofibrant, dg-category $\GSn^{\on{rd}}$ as follows: for each edge $e$ of $\Sgh$ connecting two interior vertices, we discard the loop $l_e$. We identify the two loops in degree $1-n$ at that vertex of $\QSn$, which we denote in the following by $L,L'$, and denote the resulting degree $1-n$ loop by $L_e$. We set its differential to be
\[ d(L_e)=d(L)-d(L')\,.\]
For each edge $e$ of $\Sgh$ connecting an interior vertex with a boundary vertex, we also discard the loop $l_e$, as well as the degree $1-n$ loop at $e$.

Recall that given a dg-category $C$ with finitely many objects, we denote by $C^{\on{alg}}$ the Morita-equivalent dg-algebra of morphisms in $C$. The dg-algebra $(G(\Sgh,3)^{\on{rd}})^{\on{alg}}$, see \Cref{rem:reducedGinzburg}, is isomorphic to a relative Ginzburg algebra of a triangulated surface in the sense of \cite{Chr22}. 

The Morita-equivalent dg-algebra $(\GSn^{\on{rd}})^{\on{alg}}$ is isomorphic to the Koszul dual $A(\Sgh,n)^!$ of the dg-algebra $A(\Sgh,n)$ defined in \Cref{subsec_koszuldual}.

In particular, in \Cref{pp:covering}, 
the dual RGB algebra $A(\widetilde{\Sgh},n)^!$ is isomorphic to the generalized relative Ginzburg dg-algebra $G(\widetilde{\Sgh},n)$ associated to the $n$-angulation $\widetilde{\mathbb{A}}$ considered in \cite{Chr22}. 
\end{remark}

\begin{example}
In \Cref{ex:relGinzburg}, 
Figure~\ref{fig:Q1} and Figure~\ref{fig:Q2} shows precisely the quiver $Q(\Sgh,3)^{\on{rd}}$ and $Q(\Sgh^\flat,3)^{\on{rd}}$, respectively. 
The edge at which we flip corresponds to the vertex $1$ of $Q(\Sgh,3)^{\on{rd}}$.

In \Cref{ex:quot1}, 
the quiver $Q(\Sgh,4)^{\on{rd}}$ describing the dg-category $G(\Sgh,4)^{\on{rd}}$ is depicted in \Cref{fig:Q3}.
\end{example}

\begin{theorem}\label{thm:BGAschober}
There exists an equivalence of $\infty$-categories
\[ \glsec(\eS,\psG)\simeq \mathcal{D}(\GSn)\,,\]
with $\psG$ the $\eS$-parametrized perverse schober from \Cref{constr:Frgraph} and $\GSn$ the dg-algebra from \Cref{constr:Grgraph}.
\end{theorem}

\begin{proof}
The inclusion $\on{Exit}(\Sgh)\subset \on{Exit}(\eS)$ is final (i.e.~composition preserves limits) by Quillen's Theorem A \cite[\href{https://kerodon.net/tag/02NY}{Theorem 02NY}]{Ker}, so that
\[ \on{lim}_{\on{Exit}(\Sgh)}\psG \simeq \on{lim}_{\on{Exit}(\eS)}\psG\,.\]
We consider $\psG$ as a diagram $\on{Exit}(\eS)\to \mathcal{P}r^R_{\on{St}}$ and let $\psG^{\on{L}}$ be the left adjoint diagram of $\psG$, obtained by composing $\psG$ with the equivalence of $\infty$-categories $\on{ladj}(\mhyphen)\colon \mathcal{P}r^R_{\on{St}}\simeq (\mathcal{P}r^L_{\on{St}})^{\on{op}}$. We hence have an equivalence of $\infty$-categories
\[  \on{lim}_{\on{Exit}(\Sgh)} \psG \simeq \on{colim}_{{\on{Exit}(\Sgh)}^{\on{op}}}  \psG^L\,.\]
To prove the Theorem, we may thus proceed with describing $\on{colim}_{{\on{Exit}(\Sgh)}^{\on{op}}}  \psG^L$.

The functor $\psG^L$ (restricted to ${\on{Exit}(\Sgh)}^{\on{op}}$) factors through the colimit preserving functor $\mathcal{D}(\mhyphen):\on{dgCat}[W^{-1}]\rightarrow \on{LinCat}_k\to \mathcal{P}r_{\on{St}}^L$ via the composite of the localization functor $N(\on{dgCat})\to \on{dgCat}[W^{-1}]$ with the functor $\underline{G}({\Sgh},n):\on{Exit}(\Sgh)^{\on{op}}\rightarrow \on{dgCat}$ from \Cref{constr:Grgraph}. By definition, $\GSn$ is the colimit of $\underline{G}({\Sgh},n)$.

It follows from \Cref{lem:bipartitecof} below that $\underline{G}({\Sgh},n)$ defines a cofibrant object with respect to the projective model strutures on the $1$-category $\on{Fun}({\on{Exit}(\Sgh)}^{\on{op}},\on{dgCat})$. Hence, the colimit of $\underline{G}({\Sgh},n)$ coincides with its homotopy colimit and thus describes the $\infty$-categorical colimit. We thus find the desired equivalence of $\infty$-categories
\[ \glsec(\eS,\psG)\simeq \on{colim}_{{\on{Exit}(\Sgh)}^{\on{op}}}  \mathcal{D}(\underline{G}({\Sgh},n))\simeq  \mathcal{D}(\on{colim}_{{\on{Exit}(\Sgh)}^{\on{op}}} \underline{G}({\Sgh},n)) \simeq \mathcal{D}(\GSn)\,. \]
\end{proof}

\begin{lemma}\label{lem:bipartitecof}
Let $P$ be a finite, bipartite poset with partition sets $X,Y\subset P$ and morphisms going from $X$ to $Y$. Let $C$ be a model category with finite coproducts and $F:P\to C$ a diagram valued in cofibrant objects. Assume further that, for all $y\in P$, the morphism
\[ \coprod_{\alpha:x\to y \in X/y}F(\alpha)\colon \coprod_{\alpha:x\to y \in X/y}F(x) \longrightarrow F(y)\]
is a cofibration in $C$, where $X/y=X\times_P P/y$ is the relative over-category. Then $F$ defines a cofibrant object in the category $\on{Fun}(P,C)$ with the projective model structure. In particular, the colimit of $F$ coincides with the homotopy colimit of $F$.
\end{lemma}

\begin{proof}
We need to check the right lifting property of $F$ with respect to acyclic fibrations in $\on{Fun}(P,C)$, meaning we need to solve the lifting problem
\[
\begin{tikzcd}
                                             & G \arrow[d, "\eta"] \\
F \arrow[r, "\nu"] \arrow[ru, "\mu", dashed] & H
\end{tikzcd}
\]
where $G,H\colon P\to C$ and $\eta\colon G\to H$ is a acyclic fibration, meaning that $\eta(p)$ is an acyclic fibration in $C$ for all $p\in P$. For each $x\in X$, we can use that $F(x)\in C$ is cofibrant to lift $\nu(x)$ along $\eta(x)$, defining the morphism $\mu(x)\colon F(x)\to G(x)$. Let $y\in Y$ and consider the composite morphism in $C$
\[ \xi_y\colon \coprod_{\alpha:x\to y \in X/y}F(x) \xlongrightarrow{\coprod_{\alpha:x\to y \in X/y}\mu(x)} \coprod_{\alpha:x\to y \in X/y}G(x) \xlongrightarrow{\coprod_{\alpha:x\to y \in X/y} G(\alpha)} G(y)\,.\]
Using that $\coprod_{\alpha:x\to y \in X/y}F(\alpha)$ is a cofibrantion and $\eta(y)$ a trivial cofibration, we can solve the lifting problem
\[
\begin{tikzcd}
\coprod_{\alpha:x\to y \in X/y}F(x) \arrow[r, "\xi_y"] \arrow[d, "\coprod_{\alpha:x\to y \in X/y}F(\alpha)"'] & G(y) \arrow[d, "\eta(y)"] \\
F(y) \arrow[r, "\nu(y)"] \arrow[ru, "\mu(y)", dashed]                                                         & H(y)
\end{tikzcd}
\]
defininig $\mu(y)$. Inspecting the construction, one immediately sees that these choices of $\mu(x)$ and $\mu(y)$ for $x\in X,y\in Y$ assemble into a natural transformation $\mu$. Further, by construction, $\eta(x)\circ \mu(x)=\nu(x)$ and $\eta(y)\circ \mu(y)=\nu(y)$ for all $x\in X,y\in Y$ and thus also $\eta\circ \mu =\nu$. This shows that $\mu$ is the desired lift, concluding the proof.
\end{proof}

\begin{remark}\label{rem:nCY}
We expect the $k$-linear $\infty$-category $\glsec(\eS,\psG)\simeq \mathcal{D}(\GSn)$ to admit a relative left $n$-Calabi--Yau structure in the sense of \cite{BD19} if $n$ is odd or $\Sgh$ is orientable in the sense of \Cref{def:orientable}. Locally, near a given vertex of degree $m$, there should be a relative Calabi--Yau structure on $\psG(v)$ arising from the relative Calabi--Yau structure of the derived $\infty$-category of the relative Ginzburg algebra of an $n$-gon, arising as a special case of \cite[Thm.~6.7]{Chr23}. These relative Calabi--Yau structures would then glue to a global relative Calabi--Yau structure on $\glsec(\eS,\psG)$, giving a relative and left Calabi--Yau generalization of the statement of \Cref{prop:nCY}.
\end{remark}

\subsection{Arc system kits}

For the following, we fix a weighted marked surface $\sow$ with a mixed-angulation $\AS$, dual S-graph $\Sgh$ and extended graph $\eS$, see \Cref{def:extdgraph}. We further fix a commutative ring $k$.

When considering arc system kits for an $\eS$-parametrized perverse schober $\hF$, we will always assume that the set of singularities of $\hF$ consists exactly of the interior vertices of $\eS$. We fix such an $\eS$-parametrized perverse schober $\hF$.

\begin{definition}[\cite{CHQ23}]\label{def:ask}
An \emph{arc system kit} for $\hF$ consists of
\begin{enumerate}[label=\roman*)]
\item an object $L_e\in \hF(e)$ for each non-virtual edge $e$ of $\eS$,
\item an object $L_{v,a}\in \hF(v)$ for each vertex $v$ and non-virtual incident halfedge $a$ of $\eS$,
\item an equivalence in $\hF(e)$
\[ \hF(v\xrightarrow{b}e)(L_{v,a})\simeq \begin{cases} L_e & a=b\\ 0 & \text{else}\end{cases}\]
for each pair of non-virtual halfedges $a,b$ incident to a vertex $v$ and where $b$ is part of the edge $e$.
\item an equivalence in $\hF(c)$
\[ \hF(v\to c)L_{v,b}\simeq \hF(v\to c)L_{v,a}[1-d(a,b)]\,,\]
for each virtual edge $c$ of $\eS$ incident to a vertex $v$ of weight $\infty$ and consecutive non-virtual halfedges $a,b$  (i.e.~$b$ follows $a$) incident to $v$.
\item an equivalence in $\hF(v)$
\begin{equation}\label{eq:twistofLva} T_{\hF(v)}(L_{v,b})\simeq L_{v,a}[1-d(a,b)]\,,\end{equation}
for each internal vertex $v$ of $\eS$ and consecutive internal halfedges $a,b$ incident to $v$. Here $T_{\hF(v)}$ denotes the twist functor of the spherical adjuction
\[
F_v'\coloneqq \prod_{i=1}^n\hF(v\xrightarrow{a_i}e_i)\colon \hF(v)\longleftrightarrow \prod_{i=1}^n\hF(e_i)\noloc G_v' \,.
\] If $v$ has valency $1$ with the single incident halfedge $a$, we instead require $T_{\hF(v)}(L_{v,a})\simeq L_{v,a}[1-d(v)]$ with $d(v)$ the degree of $v$.
\end{enumerate}
Note that in particular, if $v$ is $q$-valent of degree $m$, one has $T_{\hF(v)}^m(L_{v,a})\simeq L_{v,a}[q-m]$.
\end{definition}

 Arc system kits glue, in the sense that a global arc system kit amounts simply to local arc systems kits, one for each vertex of $\eS$ and its incident edges, plus the requirement that the data at the edges agrees.

\begin{definition}[\cite{CHQ23}]\label{def:positivekit}~
\begin{itemize}
\item We denote $\Ende=\on{End}_{\hF(e)}(L_e)\in \mathcal{D}(k)$ for any choice of non-virtual edge $e$ of $\eS$ (choosing a different edge yields an equivalent $k$-module) and $\Endv=\on{End}_{\hF(v)}(L_{v,a})\in \mathcal{D}(k)$ for each vertex $v$ of $\eS$ and any choice of incident non-virtual halfedge $a$.
\item We call the arc system kit of $\hF$ \textit{positive} if $\on{H}^0(\Ende)\simeq \on{H}^0(\Endv)\simeq k$, $\on{H}^i(\Ende)\simeq \on{H}^i(\Endv)\simeq 0$ for all $v\in \Sgh_0, i<0$ and finally if, for any weight $-1$ vertex $v$ with incident halfedge $a\in e\in \Sgh_1$, the $k$-vector space $\on{H}^1(\Endv)\simeq k$ is generated by the extension arising from combining the cofiber sequence
\[ L_{v,a}\xlongrightarrow{\on{unit}} \on{radj}(\hF(v\xrightarrow{a} e))\circ \hF(v\xrightarrow{a} e) (L_{v,a}) \longrightarrow T_{\hF(v)}(L_{v,a})\]
with the equivalence $L_{v,a}\simeq T_{\hF(v)}(L_{v,a})$.
\end{itemize}
\end{definition}

\begin{proposition}\label{prop:BGAschoberarcsystem}
The perverse schober $\psG$ admits a positive arc system kit.
\end{proposition}

\begin{proof}
The arc system kit is obtained as follows.
\begin{enumerate}
\item[i)] We set $L_e=\underline{k}\in \mathcal{D}(k[t_{2-n}])=\psG(e)$ to be the trivial module with cohomology $k$ for each edge $e$ of $\Sgh$.
\item[ii)] We set $L_{v,a}\in \mathcal{D}(G_v^n)\simeq \hF(v)$ for $v$ a vertex and $a$ an incident halfedge to be the $G_v^n$-module with cohomology $k$ generated by the lazy path at the vertex of $Q_v^n$ corresponding to the halfedge $a$\footnote{The corresponding object in the abelian category of $(G_v^n)^{\on{alg}}$-modules is also called the simple module associated with the vertex of $Q_v^n$.}.
\item[iii)] Equivalences as in part iii) of \Cref{def:ask} are immediate.
\item[iv)] Equivalences as in iv) in \Cref{def:ask} can also be obtained via a direct computation.
\item[v)] Consider a vertex $v$ of $\Sgh$, with degree $m$ and valency $m$. One finds that the action of the twist functor $T_{\hF(v)}$ realizes the remaining rotational $\ZZ/m$-symmetry of the dg-algebra $G_{n,m}$. This gives rise to the equivalences as in v) in \Cref{def:ask}. We omit the details in the case where $v$ has degree $m$ and valency $r<m$, where a small computation yields the equivalences from v) in \Cref{def:ask}.
\end{enumerate}
The corresponding objects $\ASG$ consist of the $\GSn$-modules generated over $k$ by the lazy paths associated with the vertices of $\QSn$.

We find for each edge $e$ of $\Sgh$ an equivalence $\on{End}(L_e)\simeq \on{H}^*(S^{n-1})\simeq k\oplus k[1-n]$ and for any vertex $v$ of degree $m$ with incident halfedge $a$ an equivalence $\on{End}(L_{v,a})\simeq k[x_{n/m}]/(x_{n/m}^{m})$, with $|x_{n/m}|=\frac{n}{m}$. The latter equivalence follows for instance from Koszul duality. The arc-system kit is thus positive.
\end{proof}

The results of \cite{CHQ23} thus apply to study $\glsec(\eS,\psG)$. This allows us to associate a global section of $\psG$ to each graded arc in ${\bf S}$. We denote by $\A_\Sgh$ the collection of objects associated in this way with the edges of the S-graph $\Sgh$, which are canonically graded arcs. We remark that $\A_\Sgh$ forms a simple-minded collection, see \cite[Prop.~4.14]{CHQ23}. Furthermore, it is not difficult to see that the objects in $\A_{\Sgh}$ correspond to the simple $G({\bf S},n)$-modules associated with the vertices of the underlying quiver.

The simple-minded collection generates a stable subcategory $\mathcal{C}(\Sgh,\psG)\subset \glsec(\eS,\psG)$ and forms the simple objects in a bounded $t$-structure on $\mathcal{C}(\Sgh,\psG)$. Further, $\mathcal{C}(\Sgh,\psG)\simeq \D^{\on{nil}}(\GSn)$ describes the nilpotent derived $\infty$-category of $\GSn$. 

All simple tilts of this initial $t$-structure on $\mathcal{C}(\Sgh,\psG)$ correspond to flips of the S-graph $\Sgh$. The exchange graph of flips thus embeds as a type of skeleton into the space of stability conditions of $\mathcal{C}(\Sgh,\psG)$. This embedding extends by \cite[Thm.~5.4]{CHQ23} to an embedding of a space of framed quadratic differentials $\FQuad{}{\sow}$ into the space of stability conditions. This includes quadratic differentials with zeros of order $k\geq 1$ so that $k+2$ divides $n$, poles of arbitrary order $k\geq 1$, exponential singularities, and marked regular points if $n$ is even, their appearance depending on $\sow$.

Finally, we describe the stable subcategory $\mathcal{C}(\Sgh,\psG)$.

\begin{proposition}\label{prop:generatedstablecategory}
The stable subcategory $\mathcal{C}(\Sgh,\psG)\subset \glsec(\eS,\psG)$ generated by $\A_\Sgh$ describes the perfect derived $\infty$-category $\per(A(\Sgh,n))$ of the RGB algebra $A(\Sgh,n)$, see \Cref{def_gbga}.
\end{proposition}

\begin{proof}
Since $\A_\Sgh$ forms a simple-minded collection, $\mathcal{C}(\Sgh,\psG)$ admits a bounded t-structure and is thus idempotent complete. By Koszul duality, the derived endomorphism algebra of $\bigoplus_{e\in \Sgh_1}\A_e$ is given by the RGB algebra $A(\Sgh,n)$, so that $\mathcal{C}(\Sgh,\psG)\simeq \per(A(\Sgh,n))$.
\end{proof}

\begin{lemma}\label{lem:Dfin=Dnil}
Suppose that each boundary component of $\sow$ contains a singular point. Then under the equivalence of $\infty$-categories from \Cref{thm:BGAschober}, $\mathcal{C}(\Sgh,\psG)$ is identified with the finite derived $\infty$-category $\D^{\on{fin}}(\GSn)$. Thus $\D^{\on{nil}}(\GSn)=\D^{\on{fin}}(\GSn)$ as full subcategories of $\D(\GSn)$.
\end{lemma}

\begin{proof}
Since $\D^{\on{nil}}(\GSn)$ and $\D^{\on{fin}}(\GSn)$ are idempotent complete, it suffices to show that there is an object $Y\in \D^{\on{nil}}(\GSn)$ which is a generator of $\D^{\on{fin}}(\GSn)$, in the sense that any $X\in \D^{\on{fin}}(\GSn)$ vanishes if and only if $\on{RHom}_{\D^{\on{fin}}(\GSn)}(\A_\Sgh,X)\simeq 0$. For each edge $e$ of $\psG$, consider the evaluation functor $\on{ev}_e\colon  \glsec(\eS,\psG) \to \psG(e)=\D(k[t_{2-n}])$ with left adjoint $\on{ev}_e^*$. A global section $X\in \D^{\on{fin}}(\GSn)$ vanishes if and only $\on{ev}_e(X)\simeq 0$ for all $e\in \Sgh$, since the functor $\on{ev}_e$ amounts to restricting a $\GSn$-module to the vertices of $\GSn$ corresponding to $e$. Using that $k\in \D^{\on{fin}}(k[t_{2-n}])$ is a generator, it follows that $\bigoplus_{e\in \Sgh}\on{ev}_e^*(k)$ is a generator of $\D^{\on{fin}}(k[t_{2-n}])$, provided that $\bigoplus_{e\in \Sgh}\on{ev}_e^*(k)\in \D^{\on{fin}}(k[t_{2-n}])$ holds. We next give a geometric description of $\on{ev}_e^*(k)$ for each edge $e$ and conclude that it lies in $\D^{\on{nil}}(\GSn)$, which then concludes the proof.

The geometric description of $\on{ev}_e^*(k)$ is obtained in the same way as in the proof of Proposition 5.19 in \cite{Chr21b}, which corresponds to the special case of triangulated surface ($n=3$) without degree $1$ vertices in the S-graph. For each edge $e$, there is a unique graded arc $c_e$ obtained as the composite of finitely many type II segments in the sense of \cite[Section 4]{CHQ23}, such that $c_e$ intersects $e$ (in degree $0$) and such that the segments of $c_e$ each begin at a halfedge of some vertex $v$ and end at the next halfedge in the counterclockwise order of the halfedges at $v$. Note that the fact that $c_e$ has finitely many segments uses that each boundary component of $\sow$ contains a singular point. By \cite[Constr.~4.12]{CHQ23}, there is an associated global section $\A_{c_e}\in \glsec(\eS,\psG)$. Using that cones between the objects in $\A_\Sgh$ corresponding to smoothing out their intersections \cite[Lem.~4.16]{CHQ23}, it is straightforward to see that $\A_{c_e}\in \D^{\on{nil}}(\GSn)$. 

The construction of the equivalence $\A_{c_e}\simeq \on{ev}_e^*(k)$ amounts to an abstract local-to-global argument. 
The main ingredients are as follows, we refer to \cite[Prop.~5.19]{Chr21b} for more details on how to assemble these. Associated with each segment $\delta$ of $c_e$ is an object $\A_\delta$ in the $\infty$-category $\losec(\eS,\psG)$ of local sections of $\psG$, in the sense of \cite[Def.~3.14]{CHQ23}. The gluing of the segments to $c_e$ corresponds to the gluing of these local sections to produce $\A_{c_e}$. To construct the desired equivalence, one first shows that $\on{RHom}_{\glsec(\eS,\psG)}(\A_{c_e},\mhyphen)\simeq \on{ev}_e$ and then passes to left adjoints. The functor $\on{RHom}_{\glsec(\eS,\psG)}(\A_{c_e},\mhyphen)$ arises from restricting the functor $\on{RHom}_{\losec(\eS,\psG)}(\A_{c_e},\mhyphen)$ which is computed by gluing the functors $\on{RHom}_{\losec(\eS,\psG)}(\A_{\delta},\mhyphen)$, which can themselves be described in terms of suitable evaluation functors. 
\end{proof}

By \Cref{thm:nonformalgen}, there exists an equivalence between $\D^{\on{nil}}(\GSn)$ and the $\infty$-category arising from the $A_\infty$-category $\mathcal C_{\mathrm{core}}(S,n)$. We expect that similarly, $\D^{\on{fin}}(\GSn)$ is equivalent to $\mathcal C_{\mathrm{len}}(S,n)$. We thus conjecture that the equality $\D^{\on{nil}}(\GSn)=\D^{\on{fin}}(\GSn)$ holds if and only if the corresponding quadratic differentials have no second order poles. This would generalize the statement of \Cref{lem:Dfin=Dnil} which corresponds to the case of quadratic differentials without poles of order $\geq 2$.

\bibliography{biblio}
\bibliographystyle{alpha}

\end{document}